\documentclass[sn-mathphys,Numbered]{sn-jnl}

\usepackage{graphicx}%
\usepackage{multirow}%
\usepackage{amsmath,amssymb,amsfonts}%
\usepackage{amsthm}%
\usepackage{mathrsfs}%
\usepackage[title]{appendix}%
\usepackage{xcolor}%
\usepackage{textcomp}%
\usepackage{manyfoot}%
\usepackage{booktabs}%
\usepackage{algorithm}%
\usepackage{algorithmicx}%
\usepackage{algpseudocode}%
\usepackage{listings}%

\theoremstyle{thmstyleone}%
\newtheorem{theo}{Theorem}[section]
\newtheorem{lemma}[theo]{Lemma}
\newtheorem{coro}[theo]{Corollary}

\theoremstyle{thmstyletwo}%

\theoremstyle{thmstylethree}%
\newtheorem{defi}{Definition}%
\newtheorem{exa}{Example}%

\usepackage{enumerate}
\usepackage{bbm}
\newcommand{\1}{\mathbbm{1}} 

\newcommand{\bal}[1]{\begin{align*}#1\end{align*}}

\newcommand{\R}{\mathbb{R}}

\newcommand{\E}{\mathbb{E}}

\newcommand{\bP}{\mathbb{P}}
\renewcommand{\P}{\bP} 

\newcommand{\eps}{\varepsilon}
\newcommand{\Var}{\mathbb{V}\mathrm{ar}}
\newcommand{\Cov}{\mathrm{Cov}}

\newcommand{\FF}{\mathcal F}

\newcommand{\PP}{\mathcal P}

\newcommand{\wh}{\widehat}

\def\pcv{\stackrel{\scriptscriptstyle \P}{\longrightarrow}}        

\begin{document}

\title[Slim and fat trees]{Scaling limits of slim and fat trees}

\author*[1]{\fnm{Vladislav} \sur{Kargin}}\email{vkargin@binghamton.edu}

\affil*[1]{\orgdiv{Department of Mathematics and Statistics}, \orgname{Binghamton University}, \orgaddress{\street{4400 Vestal Pkwy East}, \city{Binghamton}, \postcode{13902}, \state{NY}, \country{USA}}}


\abstract{We consider Galton--Watson trees conditioned on both the total number of vertices $n$ and the number of leaves $k$. The focus is on the case in which both $k$ and $n$ grow to infinity and $k = \alpha n + O(1)$, with $\alpha \in (0, 1)$. Assuming the exponential decay of the offspring distribution, we show that the rescaled random tree converges in distribution to Aldous' Continuum Random Tree with respect to the Gromov--Hausdorff topology. The scaling depends on a parameter $\sigma^\ast$ which we calculate explicitly. Additionally, we compute the limit for the degree sequences of these random trees.}

\keywords{Random tree, Galton--Watson tree, Continuum random tree, Scaling limit, Contour process}


\pacs[MSC Classification]{Primary 60J80, Secondary 60C05 05C05}

\maketitle

\section{Introduction}
In this paper, we investigate the Galton--Watson (``GW'') trees conditioned both on the total number of vertices and on the number of leaves, when the total number of vertices is large and the number of leaves is far from its expected value. For a situation with small number of leaves, we can expect long leafless branches, while in the situation when a tree has a large number of leaves, the branches are short and bushy,  and this is why we call these trees ``slim'' and ``fat,'' respectively.

The general setup is as follows.  Let $L$ be a non-negative, integer-valued random variable with the probability mass function $\P(L = i) = w_i$, $i = 0, 1, 2, \ldots$, which is assumed to satisfy the following conditions:
\begin{equation}
\label{basicCondition}
\E L = 1, \text{ and for some $a > 0$, } \E e^{a L} < \infty. 
\end{equation}
We also assume that the distribution of $L$ is not trivial: $w_1 \ne 1$. 
 
A random tree $T$ is the family tree of a Galton--Watson process with one individual in generation 0 and with offspring of each individual distributed independently as the random variable $L$. We use ordered trees so that the information on the birth order of family tree branches is known. For a formal description, see Sect. 3 in the paper of Le Gall and Le Jan \cite{legall_lejan98}.

The probability space $\Omega$ is the space of all ordered rooted trees with the probability distribution $\PP$ induced by the distribution of random variable $L$. We write  $\Omega_n\subset \Omega$ for trees with $n$ vertices and $\Omega_{k, n} \subset \Omega_n$ for trees with $k$ leaves and $n$ vertices.  We say that  \emph{$(k, n) \in PP(w)$} for an offspring distribution $w = (w_0, w_1, \ldots)$ if the set of trees $\Omega_{k,n}$ has positive probability. In this case, both $\Omega_n$ and $\Omega_{k,n}$ are endowed with conditional probability laws $\PP_n$ and $\PP_{k,n}$ induced from law $\PP$ on $\Omega$.

Note that a probability distribution $(w_j)_{j = 0}^\infty$ can be replaced by an equivalent probability distribution $(\tilde w_j)_{j = 0}^\infty$ without changing the measure on $\Omega_n$ (see Sect. 4 in Janson \cite{janson2012}). In particular, by Remark 4.3 in \cite{janson2012}, any equivalent probability distribution is determined by its expectation, and possible expectations  are in $(0, \nu]$, where $\nu$ is a certain parameter. It is easy to check by using the definition of $\nu$ (see formulas (3.8) - (3.10) in  \cite{janson2012}) that if $(w_j)_{j = 0}^\infty$  has an exponentially declining tail then $\nu \geq 1$. Therefore, for every such  $(w_j)_{j = 0}^\infty$ we can always find an equivalent probability distribution $(\tilde w_j)_{j = 0}^\infty$ so that its expectation is $1$. In particular, the condition $\E L = 1$ in (\ref{basicCondition}) does not impose an additional restriction on measure of trees in $\Omega_n$ and $\Omega_{k, n}$ but is used only to fix a specific equivalent weight distribution. 

We consider the case  when the number of leaves $k$ approximately proportional to the number of vertices $n$, more specifically, $k = \alpha n + O(1)$, where $\alpha \in (0, 1)$. Recall that if the offspring distribution $\{w_i\}_{i = 0}^\infty$ has finite variance, then the expected number of leaves for a tree in $\Omega_n$ is $\sim w_0 n$ (see Theorem II.5.1 in \cite{kolchin84}). We are mainly interested in the case when $\alpha \ne w_0$.

\subsection{Main Results}

In Theorem \ref{theoTreeConvergence}, we will show that as $n \to \infty$, a rescaled tree in $\Omega_{k, n}$ converges in the Gromov--Hausdorff topology to a random metric space called the Continuum Random Tree (``CRT''). This space was first introduced by Aldous in \cite{aldous91a}, \cite{aldous91b}, \cite{aldous93} and for details of the definition, we refer the reader to Sect. 2 in Le Gall's survey \cite{legall2005}. In order to formulate the result precisely, we introduce some additional notation in the next subsection.

\subsubsection{$\alpha$-Shifted Offspring Distribution}  Let $\theta(t) := \sum_{i = 0}^\infty w_i t^i$, where $\{w_i\}$ is the offspring distribution, and let $\rho$ denote the radius of convergence for this series. By assumption (\ref{basicCondition}),
$\sum_{i = 0}^\infty w_i = 1$ and $\sum_{i = 1}^\infty i w_i = 1.$
Therefore, $\theta(1) =  1$ and $\theta'(1) = 1$. In particular, $1 \leq\rho \leq \infty$. The assumption $w_1 \ne 1$ implies additionally that  $w_0 \ne 0$.

Define
\begin{equation}
\label{defiPsi}
\wh \psi(t) :=  \frac{t\theta'(t)}{\theta(t) - w_0}.
\end{equation}
(Janson in \cite{janson2012} uses $\psi(t)$ to denote $\frac{t\theta'(t)}{\theta(t)}$.)
In Lemma \ref{lemmaMonotonicity}, we show that $\wh \psi(t)$ is increasing on the interval $[0, \rho)$.
Hence, we can define  
\begin{equation}
\label{defiNu}
\wh \nu := \lim_{t \uparrow \rho} \wh \psi(t)  \in [1/(1 - w_0), \infty].
\end{equation}
($\wh\nu \geq 1/(1 - w_0)$ because $\wh\psi(1) = 1/(1 - w_0)$ and $\rho \geq 1$.)

%

\begin{defi}
\label{defiAlphaShift}
Let $\alpha \in (0, 1)$, and let $w_i^\ast$ be a probability mass distribution with unit mean, such that  $w_0^\ast = \alpha$ and  $w_j^\ast = C w_j (t^\ast)^{j - 1}$ for all $j > 0$ and for some $C>0$ and $t^\ast>0$.  We call $w_0^\ast, w_1^\ast, \ldots$, the \emph{$\alpha$-shifted offspring distribution}.  Let $L^\ast$ denote a random variable that have the distribution $\P(L^\ast = j) = w_j^\ast$ for $j = 0, 1, 2, \ldots$. 
\end{defi}
In other words,  in the $\alpha$-shifted distribution, $w_0^\ast$ is set to $\alpha$ and the remaining weights are exponentially re-weighted in such a way that the result is still a probability mass distribution with unit mean. It is easy to check that $C$ and $t^\ast$ can be determined from the following equations, 
\begin{equation}
\label{equForAlphaShift}
(1 - \alpha) t^\ast \theta'(t^\ast) = \theta(t^\ast) - w_0, \text{ and }
C = \frac{1}{\theta'(t^\ast)}.
\end{equation}
In Lemma \ref{lemmaExistence}, we show that these equations are solvable if $\alpha < 1 - 1/\wh\nu$. (In particular, they are always solvable for $\alpha \leq w_0$. Intuitively,  $\alpha > 1 - 1/\wh\nu$ corresponds to the situation when the number of leaves is not sustainable by the structure of the offspring distribution.)

In Lemma \ref{lemmaVarianceAlpha}, we show that the variance of the $\alpha$-shifted distribution $w_j^\ast$ is 
\begin{equation}
\label{variance}
(\sigma^\ast)^2 = \frac{t^\ast \theta''(t^\ast)}{\theta'(t^\ast)}.
\end{equation}

\subsubsection{Rescaled Convergence in the Gromov--Hausdorff Metric}
Let us consider trees in $\Omega_{k,n}$ as metric spaces with a marked point (`root') by using the usual graph theoretic distance between vertices and extending it to the edges linearly. We will denote the distance between points $u$ and $v$ as $d(u, v)$. 

Recall the definition of the Gromov--Hausdorff metric. If $(E, \delta)$ is a metric space, then the Hausdorff distance between compact subsets of $E$ is defined as follows: 
\bal{
\delta_{\textrm{Haus}} (K, K') = \inf\{\alpha > 0: K \subset U_\alpha(K') \text{ and } K' \subset U_\alpha(K)\}, 
}
where $U_\alpha(K): = \{x \in E: \delta(x, K) \leq \alpha\}$. Then, if $T$ and $T'$ are two rooted compact metric spaces, with roots $\rho$ and $\rho'$, the Gromov--Hausdorff distance is defined as 
\bal{
d_{GH}(T, T') = \inf \Big\{\delta_{\textrm{Haus}}\big(\phi(T), \phi'(T')\big) \vee \delta\big(\phi(\rho), \phi'(\rho')\big)\Big\},
}
where the infimum is over all choices of a metric space $(E, \delta)$ and all isometric embeddings $\phi: T \to E$ and $\phi': T' \to E$. Two rooted compact metric spaces $T_1$, $T_2$ are equivalent if there is a root-preserving isometry between these two spaces and obviously $d_{GH}(T, T')$ depends only on the equivalence classes of $T$ and $T'$. It is a fact that $d_{GH}(T, T')$ defines a metric on the set of equivalence classes of rooted compact metric spaces. (For the unrooted case, see Theorem 7.3.30 in Burago et al. \cite{burago_burago_ivanov2001}. The proof for the rooted case can be obtained after some adjustments.) Moreover, the set of isometry classes of real trees equipped with the Gromov--Hausdorff metric a complete and separable metric space. See Theorem 1 in a paper of Evans et al.  \cite{evans_pitman_winter2006}.

Recall that we say that $(k, n) \in PP(w)$ for an offspring distribution $w = (w_0, w_1, \ldots)$ if the set of trees with $n$ vertices and $k$ leaves has a positive probability under the probability law  $\PP$ induced by this offspring distribution. We have the following convergence result. 

\begin{theo}
\label{theoTreeConvergence}
 Let the offspring distribution $w = (w_j)_{j = 0, 1, \ldots} $ satisfy Condition (\ref{basicCondition}), and assume that $\big((k_i, n_i)\big)_{i = 1}^\infty$ is a sequence in $PP(w)$ such that $n_i \to \infty$  and $k_i = \alpha n_i + O(1)$ where  $\alpha  \in (0, 1 - 1/\wh \nu)$. Let $T_{n_i}$  be a tree chosen in $\Omega_{k_i, n_i}$ according to the probability law $\PP_{k_i, n_i}$. 
Then, 
\bal{
\frac{1}{\sqrt{n_i}} T_{n_i} \to \frac{2}{\sigma^\ast} \mathcal{T}, 
}
where $\mathcal{T}$ is the Aldous continuum random tree and the convergence is in distribution with respect to the Gromov--Hausdorff topology on the space of compact metric spaces. 
\end{theo}

\subsubsection{Convergence of Contour Processes}
Theorem \ref{theoTreeConvergence} follows from a stronger result that asserts the convergence of the rescaled contour process associated with the tree $T_n \in \Omega_{k, n}$ to a Brownian excursion. For details of this implication, see the proof of Theorem 2.5 in Le Gall's survey \cite{legall2005}. The key fact is Lemma 2.4 in \cite{legall2005}, which states that if $T$ and $T'$ are two real trees with contour functions $C(x)$ and $C'(x)$ then $d_{GH}(T, T') \leq 2 \|g - g'\|$, where $\|g - g'\|$ stands for the uniform norm of $g - g'$.

Before stating the stronger result, let us define the contour and related processes for an ordered rooted tree. While we need only the contour process to formulate the result, the other processes will be useful in the proof. 

Let $T$ be an ordered tree with $n$ vertices. 

\begin{figure}[htbp]
\centering
              \includegraphics[width=0.9\textwidth]{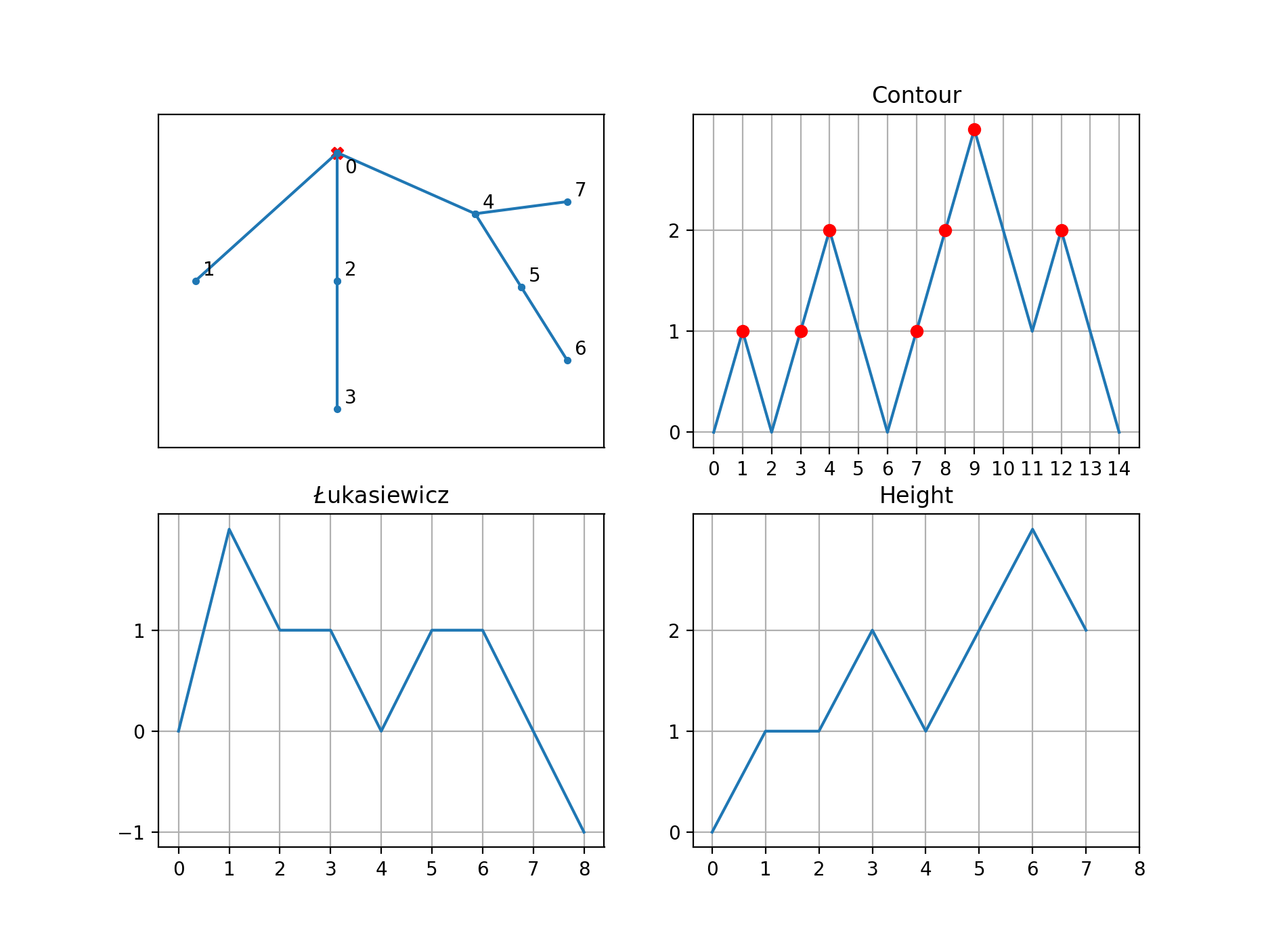}
              \caption{The contour, \L ukasiewicz, and height processes for a tree on 8 vertices. The red points $\big(m(l), H_n(l)\big)$ are an embedding of the height process in the contour process (see Definition  \ref{defi_ml}).}
              \label{figTree}
\end{figure}

\emph{The depth-first search (DFS)}: The sequence $f_i$, with  $i = 0, \ldots, 2n - 2,$ is a sequence of vertices, defined as follows: (i) $f_0 = root$; (ii) given $f_i = v,$ choose, if possible, the left-most child $w$ of $v$ that has not already been visited and set $f_{i +1} = w$. If this is not possible, let $f_{i + 1}$ be the parent of $v$, (iii) the sequence ends when $f_{2n - 2} = root$ and all vertices has been visited at least once.      

\emph{The contour process} $C_n(i)$, $i = 0, \ldots, 2n - 2,$ is defined as follows:
\begin{equation}
\label{contour_process}
C_n(i) = d(root, f_i) = \text{height of } f_i, \text{ the i-th vertex in the DFS}.
\end{equation}
Between the integer points, the path $C_n(t)$ is defined by linear interpolation. See Fig. \ref{figTree} for an illustration. 

We also define 
\emph{the \L ukasiewicz path process} $S_n(i)$, for $i = 0, \ldots, n$. (In the paper of Marckert and Mokkadem \cite{marckert_mokkadem2003}, this process is called the Depth First Queue Process (DFQP); we follow the terminology in the book of Flajolet and Sedgewick \cite{flajolet_sedgewick2009}.) Assume that the vertices in a planar tree are denoted $v_j$, $j = 0, \ldots, n - 1$, in the order of their appearance in the Depth First Search (and without repetition, unlike $f_i$). Let $\xi_i$, $i = 1, \ldots, n$, be the number of children of vertex $v_{i - 1}$.  Then, $S_n(0) = 0$, and
\begin{equation}
 \label{Lukasiewicz_process}
S_n(j) = \sum_{i = 1}^{j}(\xi_i - 1), \text{ for } 1\leq j \leq n.
\end{equation}
 Similar to the contour process, we extend the \L ukasiewicz process to the intervals between integers by linearity. For illustration, see Fig. \ref{figTree}. 
 
 A \L ukasiewicz path, by definition, is a piecewise continuous map from $[0, n] \to \R$ that has the following properties: its increments $S_n(j+1) - S_n(j) \in \{-1, 0, 1, \ldots\}$ for all integer $0 \leq j \leq n - 1$; it is non-negative, $S_n(j) \geq 0$, for every $0 \leq j \leq n - 1$; $S_n(0) = 0$, and  $S_n(n) = - 1$. The construction  (\ref{Lukasiewicz_process}) maps an ordered rooted tree with $n$ vertices to a \L ukasiewicz path. In fact, every \L ukasiewicz path can be obtained in this way, and therefore this construction gives a bijective map from the set of ordered trees $\Omega_n$ to the set of \L ukasiewicz paths on interval $[0, n]$.

Finally, we define the \emph{height process} $H_n(i)$, $i = 0, \ldots, n - 1$,  as
\begin{equation}
\label{height_process}
H_n(i) = d(root, v_i) = \text{height of } v_i.
\end{equation}

Now let us formulate the key result. In the theorem below, the convergence of stochastic processes is the weak convergence with respect to the space of continuous functions on $[0, 1]$ with the topology induced by the uniform norm: $\|f\|_{\infty} = \sup_{t \in [0,1]} f(t)$.

\begin{theo}
\label{theoConvergenceToExcursion}
  Let the offspring distribution $w = (w_j)_{j = 0, 1, \ldots} $ satisfy Condition (\ref{basicCondition}), and assume that $\big((k_i, n_i)\big)_{i = 1}^\infty$ is a sequence in $PP(w)$ such that $n_i \to \infty$  and $k_i = \alpha n_i + O(1)$ where  $\alpha  \in (0, 1 - 1/\wh \nu)$. Let $T_{n_i}$  be a tree chosen in $\Omega_{k_i, n_i}$ according to the probability law $\PP_{k_i, n_i}$ and let $C_{n_i}$ be the contour process for this tree. Then, as $n_i \to \infty$,  
\bal{
\bigg( \frac{C_{n_i}(2n_i t)}{\sqrt{n_i}}\bigg)_{t \in (0, 1)} \xrightarrow{\text{weakly}} \frac{2}{\sigma^\ast} e(t),
}
where $e(t)$ is a standard Brownian excursion, and $\sigma^\ast$ is as defined in (\ref{variance}). 
\end{theo}


This theorem implies convergence results for continuous functionals on trees, in particular, the following result about the height of trees. 
\begin{coro}
Let $h_{n_i}$ be the height of a random tree $T_{n_i} \in \Omega_{k_i, n_i}$, and suppose that the assumptions of Theorem \ref{theoConvergenceToExcursion} are satisfied. Then, 
 the distribution of $\sigma^\ast h_{n_i} / (2\sqrt{n_i})$ converges to the distribution of the maximum of a standard Brownian excursion. 
\end{coro}
\begin{proof}
The height of tree $T_{n_i}$ corresponds to the maximum of the tree's contour process. The maximum is a continuous functional on  functions in $C[0, 1]$ with the uniform metric. Therefore, by Theorem \ref{theoConvergenceToExcursion}, the distribution of the maximum for the re-scaled contour process $C_n$ converges to the distribution of the maximum of a standard Brownian excursion. 

\end{proof}


 
 \subsubsection{Convergence of the Degree Sequence}
Now, let $n_j(T)$ denote the number of vertices of out-degree  $j$ in tree $T$. (Recall that the out-degree of a vertex in a rooted tree is the number of edges that lead away from the root. For Galton--Watson trees it is the number of children of an individual represented by the vertex.) We say that the sequence $(n_0(T), n_1(T), \ldots)$ is the 
\emph{degree sequence} of tree $T$. A probability measure on trees $\Omega_{k, n}$ induces a probability measure on degree sequences. By definition, $n_0(T) = k$ for a tree $T \in \Omega_{k, n}$. In addition, $\sum_{i = 0}^\infty n_i(T) = n$, where $n$ is the number of vertices in $T$, and therefore,  $\{\frac{n_i(T)}{n }\}_{i = 0}^\infty$ is a random probability measure.

\begin{theo}
\label{theoMostLikelyType}
 Let the offspring distribution $w = (w_j)_{j = 0, 1, \ldots} $ satisfy Condition (\ref{basicCondition}), and assume that $(T_n \in \Omega_{k, n})$ is a sequence of random trees with $k \to \infty$, $n \to \infty$,  and $(k,n)\in PP(w)$. Suppose that $k = \alpha n + O(1)$ as $n \to \infty$, and $\alpha \in (0, 1 - 1/\wh \nu)$. Then, almost surely, $\frac{n_j(T_n)}{n}\to w_j^\ast$  for $j = 0, 1, 2, \ldots$, where $\{ w_j^\ast\}$ is the $\alpha$-shifted offspring distribution.
\end{theo}

\subsubsection{Examples}
Here are some explicitly calculated $\alpha$-shifted distributions together with the parameter $(\sigma^\ast)^2$. 
\begin{exa}[GW trees with geometric distribution]
\end{exa}
Consider GW trees with the geometric offspring distribution $w_j = 2^{-(j +1)}$, $ j = 0, 1, \ldots $ conditioned to have $n$ vertices. This distribution is an exponential rescaling of the weight sequence $w_j = 1$, where $ j = 0, 1, \ldots $,  and therefore these trees have the uniform distribution on $\Omega_n$. (This is Example 10.1 in Janson \cite{janson2012}.)
 For this distribution, we calculate $\theta(t) = (2 - t)^{-1}$ and $\wh\nu = +\infty$. Therefore, equations in (\ref{equForAlphaShift}) are solvable for every $\alpha \in (0, 1)$. The solution is $t^\ast = 2\alpha$ and $C = 4(1 - \alpha)^2$. This implies $w_0^\ast = \alpha$ and $w_j^\ast = \alpha^{j - 1} (1 - \alpha)^2$ for $j = 1, 2, \ldots$.  The variance is $(\sigma^\ast)^2 = 2\alpha/(1 - \alpha)$. Note that as $\alpha$ (i.e., the limiting ratio of the number of leaves to the number of vertices) increases, the variance parameter of the $\alpha$-shifted offspring distribution, $(\sigma^\ast)^2$, also increases. In particular, we should expect that the height of the random tree in $\Omega_{k,n}$ decreases. (A tree with many leaves is ``fat'' and ``short''.)

\begin{exa}
[Unary-Binary GW trees]
\end{exa}
Consider unary-binary GW trees with weight sequence $w_0 = p$, $w_1 = 1 - 2p$, $w_2 = p$, and $w_i = 0$ for all other $i$, where $0 < p < 1/2$.  In this case, $\wh \nu = 2$ and equations  (\ref{equForAlphaShift}) are solvable for every $\alpha \in (0, \frac{1}{2})$. (It is not possible to solve these equations for $\alpha > 1/2$ because  the number of leaves and the number of vertices are related by the inequality $n \geq 2 k - 1$ in the unary-binary trees, with the equality reached for binary trees.)  The setup of this example includes Examples 10.4 and 10.5 in Janson \cite{janson2012}.  After $\alpha$-shifting we get $w_0^\ast = w_2^\ast = \alpha$, $w_1^\ast = 1 - 2\alpha$, with variance $(\sigma^\ast)^2 = 2\alpha$. We can see that, as $\alpha$ increases, the variance also increases, as in the previous example. However, in this case $\alpha$ is bounded from above, therefore variance of the offspring distribution is bounded from above, and the tree cannot be made arbitrarily short.

\subsubsection{Algorithm for Generating Trees in $\Omega_{k, n}$}
The proof of Theorems \ref{theoConvergenceToExcursion} and \ref{theoMostLikelyType} is based on a new algorithm that generates trees in $\Omega_{k, n}$ with weight sequence $w_1, w_2, \ldots$. In order to explain the algorithm, we introduce a useful map between ordered trees and balls-in-boxes allocations. The general idea of the algorithm is to generate a suitable random allocation and then recover a random tree from this allocation. 

First, we define a map $\sigma$ that sends an ordered tree on $n$ vertices to the sequence of the vertex offspring sizes. In this sequence the vertices are ordered by the DFS, but the repeats of the vertices are omitted:
\begin{equation}
\label{defiSigma}
\sigma: T \to \sigma(T) = (\xi_1, \ldots, \xi_n).
\end{equation} 
For example in Fig. \ref{figTree}, $\sigma(T) = (3, 0, 1, 0, 2, 1, 0, 0)$. 

This sequence can be identified with an allocation of $n - 1$ balls (corresponding to edges) in $n$ boxes (corresponding to vertices). The set of all possible allocations is denoted $B_{n - 1, n}$.
The map $\sigma: \Omega_n \to B_{n-1, n}$ is injective but not surjective. 


By using the cycle lemma\footnote{The cycle lemma appears first in a paper of Dvoretzky and Motzkin \cite{dvoretsky_motzkin1947} and re-appears in different forms in many other papers. See Dershowitz and  Zacks \cite{dershowitz_zaks1990} and Note I.47 on p.75-76 of the book by Flajolet and Sedgewick \cite{flajolet_sedgewick2009} for an overview of the literature. The lemma is closely related to Vervaat's transform for the Brownian Bridge (see \cite{vervaat79}) and is often called the discrete Vervaat transform (see Pitman \cite{pitman2002}, Exercise 1 in Sect. 5.1). 
}, one can establish the following result.
\begin{lemma}[Corollary 15.4 in \cite{janson2012}]
\label{lemmaCyclicShift}
If $(y_1, \ldots, y_n) \in B_{n - 1, n}$, then exactly one of the $n$ cyclic shifts of $(y_1, \ldots, y_n)$ is the out-degree sequence $\sigma(T)$ of a tree $T \in \Omega_n$. 
\end{lemma}


\begin{figure}[htbp]
\centering
              \includegraphics[width=\textwidth]{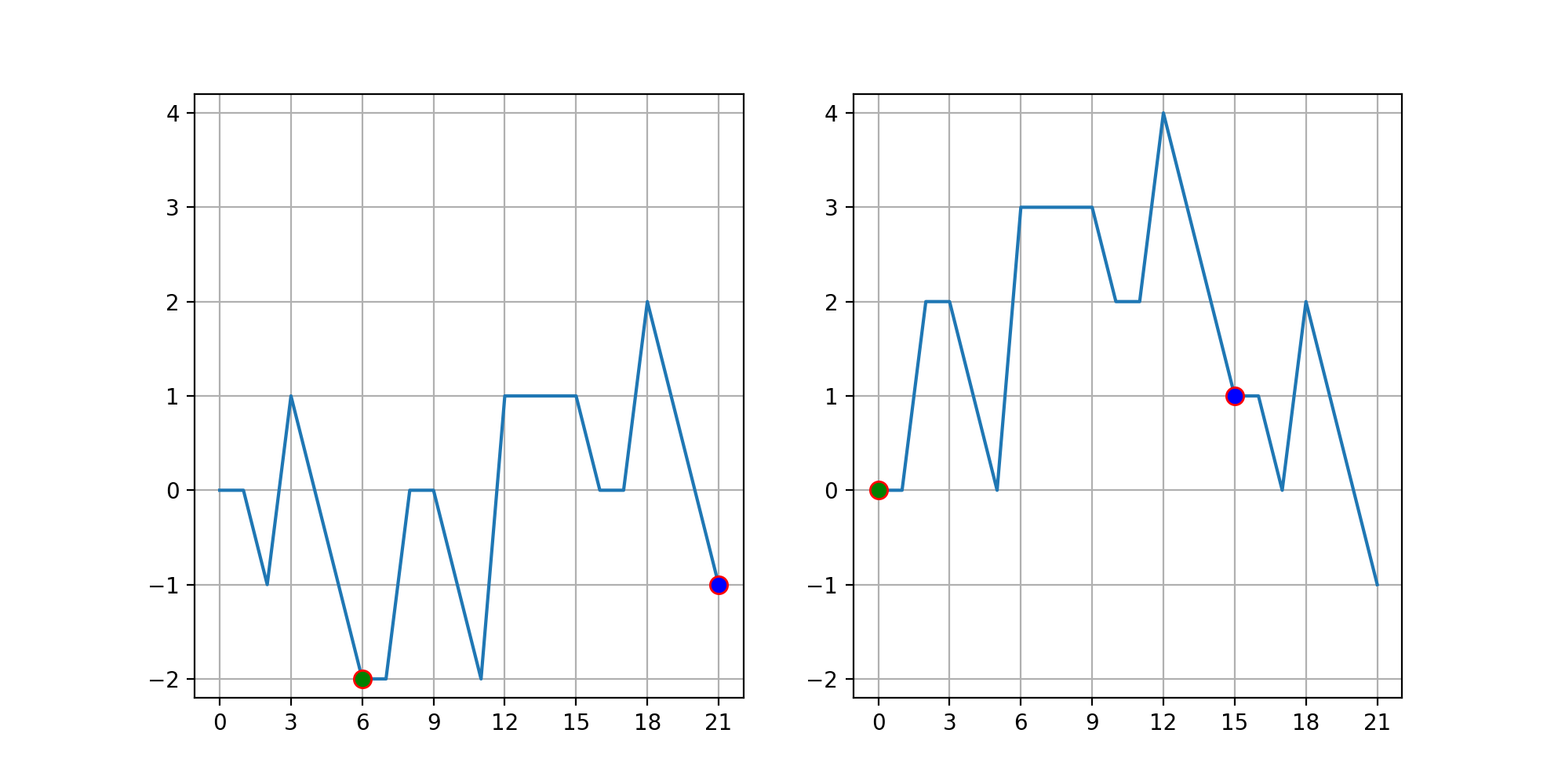}
              \caption{An example of the cycle shift that sends a random walk bridge to a \L ukasiewicz path.}
              \label{figCyclicShift}
\end{figure}

 The relevant cyclic shift can be identified as follows. First note that the sequence $y_1, \ldots, y_n$ corresponds to a path $W_0 = 0$, $W_j = \sum_{i = 1}^j (y_i - 1)$, $j = 1, \ldots, n$, and that for this path $W_n = -1$. Let $m = \min_{ 0 \leq j \leq n} W_j$, and let $\wh j$ be the first time when the path achieves this minimum. Then the relevant cyclic shift of $W_j$ is the path for the random walk with the steps given by $(y_{\wh j+1} - 1, \ldots, y_n - 1, y_1 - 1, \ldots y_{\wh j} - 1)$. It is easy to see that the resulting path is a \L ukasiewicz path and, therefore, it corresponds to a tree $T \in \Omega_n$.
 
A cyclic shift is illustrated in Fig. \ref{figCyclicShift}.

Now we define an algorithm for the generation of trees in $\Omega_{k, n}$.
Define the probability mass function $\wh w_j^\ast$ as follows:

\begin{equation}
\label{defiHatShifted}
\wh w_0^\ast := 0; \, 
\wh w_j^\ast := \frac{w_j^\ast}{1 - \alpha}, \text{ for } j = 1, 2, \ldots,
\end{equation}
where $w_j^\ast$ is the $\alpha$-shifted distribution from Definition \ref{defiAlphaShift}.\footnote{In the proofs of our theorems the parameter $\alpha$ is given to us and satisfy the assumption that $k = \alpha n + O(1)$ as $n \to \infty$. If $\alpha$ is not known and the task is simply to generate the tree with $k$ leaves and $n$ vertices, then one can take $\alpha = k/n$.}

Let the random sequence $\xi'_1, \ldots, \xi'_{n - k}$ have the following distribution 
\begin{equation}
\label{distrXiPrime}
\P(\xi'_1 = j_1; \ldots; \xi'_{n - k} = j_{n - k}) = 
\begin{cases}
\frac{1}{Z}\prod_{i = 1}^{n - k} \wh w_{j_i}^\ast, & \text{if } \sum_{i = 1}^{n - k} j_i = n - 1, 
\\
0, & \text{otherwise.}
\end{cases}
\end{equation} 
where $Z$ is a normalizing constant. In other words, $\xi'_1, \ldots, \xi'_{n - k}$ has the distribution of i.i.d. random variables with distribution $\wh w_{j}^\ast$, conditioned to have the sum $n - 1$.

\section*{Algorithm A}
\begin{enumerate}
\item Generate a sample of $\xi'_1, \ldots, \xi'_{n - k}$ with distribution as in (\ref{distrXiPrime}). 
\item Add $k$ zeros to the sequence $\xi'_i$ and let $\xi''=(\xi''_i)_{i\leq n}$ be a uniform permutation of the resulting sequence.
\item Apply the (unique) cyclic shift to the resulting sequence $(\xi''_i)_{i\leq n}$ so that the result is in the image of the map $\sigma$ defined in (\ref{defiSigma}).  Let the result be denoted $\xi_1, \ldots, \xi_n$.
\item Obtain a tree $T \in \Omega_{k, n}$ as $\sigma^{-1}(\xi_1, \ldots, \xi_n)$. 
\end{enumerate}

 It is assumed that the first step of this algorithm is implemented efficiently as described in Devroye \cite{devroye2012}.


\begin{theo}
\label{theoAlgorithm}
 Let the offspring distribution $w = (w_j)_{j = 0, 1, \ldots}$ satisfy Condition (\ref{basicCondition}). Suppose that $\alpha \in (0, 1 - 1/\wh \nu)$ and $(k, n) \in PP(w)$.  Then, algorithm A produces a tree $T \in \Omega_{k,n}$ with the probability distribution $\PP_{k, n}$ corresponding to weights $w_0, w_1, \ldots.$ If $k = \alpha n + O(1)$,  then the expected running time of the algorithm is $O(n)$, as $n \to \infty$.
\end{theo}

Note that the essential corollary of this theorem is that for a random tree $T \in \Omega_{k, n}$ with the law $\PP_{k, n}$, the distribution of $\sigma(T)$ equals the distribution of the random sequence $(\xi_1, \ldots, \xi_n)$ produced by the algorithm. This gives us the ability to study the \L ukasiewicz  and related processes for the random trees in $\Omega_{k, n}$.

We will first prove Theorem \ref{theoAlgorithm}. After this theorem is proved, the general strategy is as follows. We use the fact that the sequence $\big(\xi''_i\big)_{i = 1}^n$ from Algorithm A is exchangeable and apply tools from Billingsley \cite{billingsley1968} to show that the partial sums of this sequence converge (after rescaling) to the Brownian bridge. Then, properties of the discrete Vervaat's transform ensure that the partial sums of $\big(\xi_i\big)_{i = 1}^n$ converge after rescaling to a Brownian excursion. Then, we adapt tools from Marckert and Mokkadem \cite{marckert_mokkadem2003} in order to show that this convergence implies the re-scaled convergence of the height and the contour processes to a Brownian excursion. The adaptation is necessary, because the sequence of i.i.d. random variables underlying $\big(\xi_i\big)_{i = 1}^n$ is conditioned not only on the sum being equal to $-1$ but also on the number of increments $-1$ being equal to $k$. The changes to the argument are straightforward but we give the argument in full detail for the convenience of interested reader. 

Together, this gives a proof of Theorem \ref{theoConvergenceToExcursion} and, consequently, of Theorem 
\ref{theoTreeConvergence}. 

For the proof of Theorem \ref{theoMostLikelyType}, we again start with results of Theorem \ref{theoAlgorithm}. In this case, we note that it is enough to show that $n_{\xi'}(i)$, the number of occurrences of increment $i$ in the sequence $\big(\xi'_i\big)_{i = 1}^{n - k}$, after normalization by $(n - k)^{-1}$ converges  almost surely to $\wh w^\ast_i$. In  order to show this, we use an idea from a paper of Broutin and Marckert\cite{broutin_marckert2014}. Roughly speaking, the idea is that under certain conditions the events which are rare under a random walk with i.i.d. increments $\hat \zeta_j$, remain rare, even if a condition is imposed on the random walk (for example, if the variables $\hat \zeta_j$ are required to add up to a specific constant). By using this trick, we can extend the almost sure convergence for $n_{\hat \zeta}(i)$, where $(\hat \zeta_j)_{j = 1}^{n - k}$ are i.i.d., to the almost sure convergence of $n_{\xi'}(i)$. This allows us to prove Theorem \ref{theoMostLikelyType}

The rest of the paper is organized as follows. In Sect. \ref{section_literature} we describe the relation of our results to previous work.  In Sect. \ref{section_proofs}, we give detailed proofs. In Sect. \ref{section_conclusion},  we discuss some remaining problems. Appendix is devoted to the proof of an auxiliary local limit theorem.

%

\subsection{Relation to Previous Work}
\label{section_literature}

Excellent reviews  of the vast literature about the asymptotic behavior of finite random trees can be found in a book \cite{flajolet_sedgewick2009} by Flajolet and Sedgewick, and in review articles \cite{legall2005} by Le Gall and \cite{janson2012} by Janson.

From early on, it was observed that different families of random trees often have some similar macroscopic properties. For example, the tree height distribution has the same functional form for many tree families, although the distribution parameters could depend on the specific choice of the family. 

In a sequence of papers, \cite{aldous91a}, \cite{aldous91b}, \cite{aldous93}, Aldous provided an explanation for this phenomenon by showing that certain sequences of finite random trees converge under rescaling in the Gromov--Hausdorff topology to  universal objects that he called the \emph{continuum random trees (CRTs)}. In particular, he showed that Galton--Watson trees with a finite-variance offspring distribution converge under rescaling to a special CRT, which he called the \emph{Brownian CRT}. (We will sometime refer to this limit as  the Aldous CRT.) The tree is characterized by the property that its contour process is the Brownian excursion. Aldous' proof of this result is based on graph theoretical arguments which showed that all such Galton--Watson trees converge under rescaling to the same limit independent of the specific offspring distribution.  Then proof was concluded by an investigation of a particular case of GW trees with geometric offspring distribution. 

 Marckert and Mokkadem in \cite{marckert_mokkadem2003} gave a different proof of this result under the condition that the offspring distribution has exponentially declining tails. In their proof, they establish the convergence (under rescaling) of the \L ukasiewicz process to the Brownian excursion and then, show that the paths of the height process must be uniformly close to the paths of the \L ukasiewicz process. For this implication, they use a relation of Galton--Watson trees to queues and an explicit formula for the height process in terms of the \L ukasiewicz process.  Then one can handle the height process by using methods from the classic theory about ladder heights (Feller \cite{feller71}, Chapter XII). Finally, the rescaled convergence of the height process implies the rescaled convergence of the contour process. 
 
 Duquesne in \cite{duquesne2003} extends Aldous' result to Galton--Watson trees that may have infinite variance. He relies heavily on the machinery developed in papers of Duquesne and Le\phantom{ }Gall \cite{duquesne_legall2002} and of Le\phantom{ }Gall and Le\phantom{ }Jan \cite{legall_lejan98}.  While their approach is similar to the approach of Marckert and Mokkadem in \cite{marckert_mokkadem2003}, however, in the case of infinite variance  the \L ukasiewicz process converges under rescaling to an excursion of a L\'evy process, and it is not immediately clear what is the candidate for the rescaled limit of the associated height process.  The authors of these papers defined a height process for the L\'evy process and then, showed that the height process of the GW trees converges under rescaling to this continuous-time height process.  In this way, they showed that Galton--Watson trees with infinite variance offspring distribution converge under rescaling to \emph{L\'evy CRTs}, encoded by these height processes. (There are several such CRTs, depending on the parameter of the L\'evy process.)
 
  An alternative proof of Duquesne's theorem was given by  Kortchemski in \cite{kortchemski2013}. The proof is still based on the results in Duquesne and Le\phantom{ }Gall \cite{duquesne_legall2002}. However, Kortchemski uses the fact that the probability measure of events conditioned on the size of a GW tree equal to $n$ is absolutely continuous with respect to the probability measure of the events conditioned on the size of the tree being at least $n$, and this type of conditioning was already handled in Sect. 2.5 of \cite{duquesne_legall2002}.  A similar absolute continuity relation can be obtained for the continuous limit objects, and Kortchemski concludes the proof by showing the uniform convergence of the Radon--Nikodym densities. 
  
  By a similar method, in \cite{kortchemski2012}, Kortchemski proved that if one conditions a Galton--Watson tree on the number of leaves, instead of the total number of vertices, then one still has the convergence under rescaling to the  Brownian CRT or a L\'evy CRT. The key additional ingredient in the proof is the fact that a GW tree with $k$ leaves has $k/w_0$ total vertices with high probability. He also proved an extension of this result to the case when one conditions the GW tree on the number of vertices that have degrees in a given non-empty set of integers.


The applicability of these universality results has been significantly enlarged in \cite{haas_miermont2012} by Haas and Miermont. They considered a big class of trees with Markov branching property. In random trees with this property, the subtrees above some given height are independent with a law that depends only on the tree size, and the size is measured either as the number of leaves or vertices.  Their result is that these trees converge under rescaling to universal random objects, called self-similar fragmentation trees. 

In particular, Haas and Miermont show  in \cite{haas_miermont2012}  that the convergence under rescaling holds in the case when the Galton--Watson trees are conditioned either on the total progeny,  or on the number of leaves (but not both).  In \cite{rizzolo2015}, Rizzolo  used the method of \cite{haas_miermont2012} to show the convergence under rescaling of the Galton--Watson trees conditioned on the number of vertices that have degrees in a specified set. The proof of these results is based on the theory of fragmentation processes, and this gives a completely different approach to results by Aldous, Duquesne, Le Gall, Le Jan, and Kortchemski. 

In this paper, we use the approach based on random walks and their continuous-time limits and not on fragmentations processes. 


The difference of this paper from earlier studies described above is that we focus on the situation when a Galton--Watson random tree is conditioned on both the number of  total progeny and the number of leaves. In particular,
 we consider a situation in which these two requirements are at odds with each other. In a sense, the obtained results are about an atypical event since they are concerned with the structure of trees under an additional condition super-imposed on the usual conditional Galton--Watson tree.

A similar approach is taken in papers \cite{labarbe_marckert2007} by Labarbe and Marckert,  and \cite{broutin_marckert2014} by Broutin and Marckert. The paper \cite{labarbe_marckert2007}  considers Bernoulli random walks (that is, walks with increments in $\{-1, +1\}$) and associated processes, -- bridges, meanders and excursions, -- conditioned on a given number of peaks. A well-known bijection maps the contour process of an ordered tree to a Bernoulli excursion, and the results in \cite{labarbe_marckert2007} imply the convergence of uniformly distributed ordered trees with a given number of leaves to the Aldous continuum random tree  (see Sect. 1.1 in their paper).\footnote{In addition, the paper \cite{labarbe_marckert2007} studies the counting process for leaves under DFS exploration of the tree. For extension of these results to GW trees, see a recent paper of Th\'evenin \cite{thevenin2020}.} This is a particular case of our Theorem \ref{theoTreeConvergence} for the case when weights $w_j$ are geometric, $w_j = 2^{-j - 1}$.

For ordered trees with non-uniform distributions,  the contour process is still an (interpolated) excursion with increments $\{-1, +1\}$ but its distribution is more complicated and its rescaled convergence is difficult to establish directly. So, instead of the contour process, we use the \L ukasiewicz process, that has a more tractable distribution, in order to prove the rescaled convergence to the Brownian excursion. Then, we infer that this implies the convergence for the contour process as well, by applying the techniques developed by Marckert and Mokkadem in \cite{marckert_mokkadem2003}.
This allows us to derive the convergence under re-scaling for a broader class of random trees than that considered in the paper of Labarbe and Marckert \cite{labarbe_marckert2007}. 

From the technical viewpoint, the methods by which we obtain the re-scaled convergence of the \L ukasiewicz process to the Brownian excursion, also differ from methods in \cite{labarbe_marckert2007}.
In particular, in \cite{labarbe_marckert2007} the authors use combinatorics to explicitly evaluate the number of simple walk paths with a given number of peaks between two lattice points. Then, they use the Stirling formula and obtain the finite-dimensional distributions of the limit by calculating the asymptotics of these explicit expressions. The combinatorics at the heart of this approach becomes significantly harder for paths with a more general set of increments, and we choose to use a different approach. 

For tightness, \cite{labarbe_marckert2007}  uses a combinatorial trick that maps simple walks with increments $\pm 1$ to walks with $\{0, 1\}$ increments  in a non-trivial fashion: peaks and valleys of the original walk path correspond to increment $1$, and the other points between steps of the original path to increment $0$. This map is injective and converts conditioning on the number of peaks to conditioning on the final point of the walk, which is easier to handle. Then, it is noted that a formula for the inverse map involves summation of independent geometric random variables, and this fact is used in order to show that the tightness also holds for the original random walk. Additional conditioning on the final point of the original random walk (needed for bridges and excursions) is handled by using the explicit formulas for the number of paths with a given number of peaks. 

This argument is hard to generalize to walks with more general increments (and so to \L ukasiewicz walks) because of the absence of the required combinatorial trick. 

Instead, we use the connection of discrete walks with the balls-in-boxes model, as described by Kolchin \cite{kolchin84} and  Janson \cite{janson2012}, and with exchangeable random variables. This allows us to apply the results in Billingsley's book \cite{billingsley1968}  that establish the re-scaled convergence for partial sums of a certain class of exchangeable random variables.

In a paper of Broutin and Marckert, \cite{broutin_marckert2014}, the entire degree sequence of a tree is specified.\footnote{See also an application of this approach to random forests in Lei \cite{lei2019}, and an extension in a paper of Ojeda et al. \cite{oht2021} where different assumptions about the degree sequence are used.} More precisely, let $d_i(\tau)$ denote the number of vertices of out-degree $i$ in the tree $\tau$, and let $\vec d(\tau) = \{d_i(\tau)\}_{i = 0}^\infty$ be the degree sequence of $\tau$. Also, let $|\vec d|=\sum_{i = 0}^\infty d_i $ and $\Delta(\vec d) =\max \{d_i\}$ be the size and the maximum degree of the tree. Assume that $\vec d^{(1)}, \vec d^{(2)}, \ldots $ be a sequence of degree sequences so that $|\vec d^{(j)}| \to \infty $, $\Delta(\vec d^{(j)}) = o(\sqrt {|\vec d^{(j)}|})$, and the measure $\vec d^{(j)}/|\vec d^{(j)}|$ converges to a probability measure $\mu_0, \mu_1, \ldots$ that has finite second moment. 

Take the uniform measure on the set of trees with degree sequence $\vec d^{(j)}$ and let the correspondent random tree be denoted $T_j$. Broutin and Marckert showed that the sequence of trees $T_j$ converges under rescaling to the Brownian CRT in the Gromov--Hausdorff topology. The proof is based on the approach of Marckert and Mokkadem \cite{marckert_mokkadem2003} and it is  similar to our approach. In particular, their proof of the convergence of \L ukasiewicz process uses Billingsley's theorem 24.1 in \cite{billingsley1968}. In their main theorem, the conditions for applicability of Billingsley's theorem hold by assumption on $\Delta(\vec d^{(j)})$.

The results of \cite{broutin_marckert2014} could be used to prove our main result, Theorem \ref{theoTreeConvergence}, provided one can show that as $n$ grows, the degree sequences of the trees in $\Omega_{n, k}$ (after normalization) converge in probability  to the $\alpha$-shifted distribution and that they satisfy the condition that $\Delta \vec d^{(j)} = o\Big(\sqrt {|\vec d^{(j)}|}\Big)$. 

In this paper we choose a different, more direct approach to the proof of the main result. However, Theorem \ref{theoMostLikelyType} contains the main ingredient of the different approach. Namely, in this theorem we show that as $n\to \infty$, the (normalized) degree sequence of a tree in $\Omega_{n, k}$ almost surely converges to the $\alpha$-shifted distribution.

%

\section{Proofs}
\label{section_proofs}
\subsection{Preliminary Results}
\label{sectionPreliminary}

We start with proving some results about the existence of the $\alpha$-shifted distribution and about its variance. 

\begin{lemma}
\label{lemmaMonotonicity}
The function $\wh\psi(t)$ in (\ref{defiPsi}) is an increasing continuous function on $(0, \rho)$. 
\end{lemma}
\begin{proof}
Let $\wh \theta(t) = \theta(t) - w_0 = \sum_{i = 1}^\infty w_i t^i$, and let 
\bal{
p_i(t) = \frac{w_i t^i}{\wh \theta(t)}  \text{ for } i = 1, 2, \ldots, 
}
for $t \in (0, \rho)$, and zero otherwise. 
 Then, for every $t \in (0, \rho)$,  $p_i(t)$ is a probability mass function (pmf) of a discrete random variable $X_t$. 
 
 By definition, $\wh \psi(t) = t \theta'(t)/\wh \theta(t) =  t \wh\theta'(t)/\wh \theta(t)$.
Then we calculate:
\bal{
t \wh \psi'(t) &= \frac{t \wh\theta'(t)}{\wh\theta(t)} + \frac{t^2 \wh\theta''(t)}{\wh\theta(t)}
 - \frac{t^2 \wh \theta'(t)^2}{\wh \theta(t)^2}
 \\
 &=  \sum_{i = 1}^\infty i p_i(t) + \sum_{i = 1}^\infty i(i - 1) p_i(t) -
 \Big(\sum_{i = 1}^\infty i p_i(t)\Big)^2 
 \\
 &= \E(X_t^2) - (\E X_t)^2 = \Var(X_t) > 0.
 } 
\end{proof}

\begin{lemma}
\label{lemmaExistence}
Let $w_0, w_1, \ldots $ be a probability distribution with unit mean and finite variance. If $\alpha < 1 - 1/\wh \nu$  with $\wh \nu$ as defined in (\ref{defiNu}), then there exist $C > 0$ and $t^\ast > 0$ that solve problem (\ref{equForAlphaShift}). 
\end{lemma}
\begin{proof}
We can write  (\ref{equForAlphaShift}) as  $C = 1/\theta'(t^\ast)$ and 
\begin{equation}
\label{problem2}
\wh \psi(t^\ast) = \frac{1}{1 - \alpha}. 
\end{equation}
By Lemma \ref{lemmaMonotonicity}, the function $\wh \psi(t)$ is continuous and increasing on the interval $(0, \rho)$, and we have: $\lim_{t \to 0} \wh \psi(t) = 1$, $\lim_{t \to \rho} \wh \psi(t) = \wh \nu$. So if $(1 - \alpha)^{-1} < \wh \nu$, then the equation (\ref{problem2}) has a solution $t^\ast \in (0, \rho)$ by the intermediate value theorem, and we can set $C = 1/\theta'(t^\ast)$ since $\theta'(t^\ast)$ is well-defined and non-zero for $t \in (0, \rho)$.  
\end{proof}

\bigskip
\begin{lemma}
\label{lemmaVarianceAlpha}
The variance of the $\alpha$-shifted distribution $\{w_j^\ast\}$, $j = 0, 1, \ldots$,  is 
\begin{equation}
\label{varianceAlpha}
(\sigma^\ast)^2 = \frac{t^\ast \theta''(t^\ast)}{\theta'(t^\ast)},
\end{equation}
 where $t^\ast$ is the solution of equations (\ref{equForAlphaShift}).
\end{lemma}
\begin{proof}
The generating function for the distribution $\{w_j^\ast\}$ is 
\bal{
\theta^\ast(t) &:= \sum_{j = 0}^\infty w^\ast_j t^j  = \alpha + C w_1 t + C w_2 t^\ast t^2 + \ldots 
\\
&= \alpha + \frac{C}{t^\ast}\big(\theta(t^\ast t) - w_0\big).
}
Then the mean is  $(\theta^\ast)'(1) = C \theta'(t^\ast) = 1$ by (\ref{equForAlphaShift}). For probability distributions with unit mean, the variance equals the second derivative of the generating function evaluated at $t = 1$. So we get that the variance is 
\bal{
(\theta^\ast)''(1) = C t^\ast \theta''(t^\ast) = \frac{t^\ast \theta''(t^\ast) }{\theta'(t^\ast)}.
}
\end{proof}

%

\subsection{Proofs of Main Results}
In this section we prove Theorems \ref{theoAlgorithm}, \ref{theoConvergenceToExcursion} and \ref{theoMostLikelyType}, about the generation algorithm, the convergence of the contour process, and the most likely type of the degree profile,  respectively.

First, we establish some useful notation and give additional details about the relation of Galton--Watson trees and balls-in-boxes allocations.

Let $w_k$, $k = 0, 1, \ldots $ be a probability mass function with unit mean. It induces distributions on trees in $\Omega_n$, as conditioned Galton--Watson trees, and on allocations of the balls-in-boxes model $B_{n - 1, n}$. [The probability of an allocation  $(y_1, \ldots y_n)$ is $\propto \prod_{i = 1}^n w_{y_i}$.]  For the injective map $\sigma$ from (\ref{defiSigma}), the push-forward probability distribution on $\sigma(\Omega_n)$ equals the distribution on $B_{n - 1, n}$ conditioned to the image of the map $\sigma$.

The following result illustrates the relation between random trees and the balls-in-boxes model. 

For a balls-in-boxes allocation $y = (y_1, \ldots, y_n)$, let $N_s(y)$ be the number of those $y_i$ in allocation $y$ that are equal to  $s$. The \emph{occupation profile} of an allocation $y$ is the sequence $\big(N_s(y)\big)_{s = 0}^\infty$. In particular, $N_0(y)$ denotes the number of empty boxes in an allocation.  

Note that a cyclic shift does not change the probability of an allocation and does not change numbers $N_s(y)$. This fact and the cyclic shift lemma (see Lemma \ref{lemmaCyclicShift}) imply the following result. Recall that $n_s(T)$ denote the number of vertices of out-degree $s$ in a random tree $T\in \Omega$. Then, the joint distribution of the degree sequence $\big(n_s(T)\big)_{s = 0}^\infty$ coincides with the joint distribution of the occupation profile $\big(N_s(y)\big)_{s = 0}^\infty$ for a random allocation $y$ in $B_{n - 1, n}$.
\begin{lemma}[Lemma II.2.2 in \cite{kolchin84}]
\label{lemmaEquivalenceTreesAndBBM}
If the set $\Omega_n$ is not empty, then 
\bal{
\P(n_s(T)= n_s; \, s = 0, \ldots, n - 1) 
= \P(N_s(y) = n_s; \, s = 0, \ldots, n - 1). 
}
\end{lemma}

Now we turn to the proofs of Theorems. 

%

\begin{proof}[Proof of Theorem \ref{theoAlgorithm}]

Since $\sigma$ is a one-to-one correspondence, for the first claim of the theorem it is enough to show that the distribution of the random sequence $(\xi_1, \ldots, \xi_n)$ produced in the third step of the algorithm equals to the distribution of $\sigma(T)$ for $T \in \Omega_{k, n}$ distributed according to $\PP_{k, n}$. 

Let $v_i$, $i = 0, \ldots, n-1$, be the vertices of an ordered rooted tree $T\in \Omega_{k, n}$ listed in the DFS order. The probability of the tree is 
\bal{
p(T) \propto \prod_{i = 0}^{n-1} w_{d(v_i)}.
}
The probability of an allocation $y = \sigma(T) = (y_1, \ldots, y_n) \in B(n - 1, n)$ is 
\begin{equation}
\label{distrOnImageSigma}
p(y) \propto \prod_{j = 1}^n w_{y_j},
\end{equation}
and the allocations are different for different $T$ by the injectivity of map $\sigma$.
If $\tau_i$ denotes a shift of allocation $y$ by $i$ positions then for all $y$ in the image of $\sigma$, the allocations $\tau_i(y)$ are different for $i = 0, \ldots, n - 1$ by the Cyclic Lemma \ref{lemmaCyclicShift} and their union covers the set 
\bal{
B_k(n - 1, n) = \{y: y \in B(n - 1, n), \, N_0(y) = k\} .
} 

Hence the distribution (\ref{distrOnImageSigma}) on the allocations in $\sigma(\Omega_{k, n})$ can be obtained provided that one can generate allocations in $B_k(n - 1, n)$ with distribution 
\begin{eqnarray}
\label{distrOnAllAllocations}
p_k(y) &\propto \prod_{i = 1}^n w_{y_j} \propto w_0^k \prod_{j = 1}^\infty w_j^{N_j(y)}
\end{eqnarray}

(The difference of (\ref{distrOnImageSigma}) and
(\ref{distrOnAllAllocations}) is that the latter is over all allocations in $B_k(n - 1, n)$ while the former is only over the allocations in $\sigma(\Omega_{k, n}) \subset B_k(n - 1, n)$. )

The distribution (\ref{distrOnAllAllocations}) coincides with the distribution of random variables $\xi_1, \ldots, \xi_n$ which are i.i.d random variables $\zeta_1, \ldots, \zeta_n$ with distribution $w_0, w_1, \ldots $ conditional that $\sum_{i = 1}^n \zeta_i = n - 1$ and that $k$ of  variables $\zeta_i$ take value $0$.  

By exchangeability of random variables $(\xi_1, \ldots, \xi_n)$, all locations for those random variables $\xi_i$ that take value $0$ are equally probable. Conditional on the choice of these locations, the remaining variables, which we denote $\xi' = (\xi'_1, \ldots, \xi'_{n - k})$, have the joint distribution  
\bal{
p(\xi') \propto \prod_{j = 1}^\infty w_j^{N_j(\xi')} \text{ for } \xi' : \sum_{i = 1}^{n - k} \xi'_i = n - 1. 
}
 
Note that the distribution of $\xi'_1, \ldots, \xi'_{n-k}$ is the same as the probability distribution of an i.i.d sequence $\zeta'_1, \ldots, \zeta'_{n-k}$, with   $\zeta'_1$ having the distribution $\wh w_0 = 0, \wh{w}_i = w_i / (1 - w_0)$ conditional that  $\sum_{i = 1}^{n - k} \zeta'_i = n - 1$. By exponential re-weighting, which is possible by Condition (\ref{basicCondition}), this distribution equals the probability distribution of an i.i.d sequence $\hat\zeta_1,  \ldots, \hat \zeta_{n-k}$ with $\hat \zeta_1$ having distribution $\wh w_0^\ast = 0, \hat w_i^\ast = w_i^\ast/(1 - \alpha)$ and conditioned to have the sum $\sum_{i = 1}^{n - k} \hat\zeta_i = n - 1$. In other words, the random variables  $(\xi'_1, \ldots, \xi'_{n - k})$ have the distribution given in (\ref{distrXiPrime}). This proves the first part of the theorem and shows the validity of the algorithm: once we generated random variables  $ (\xi'_1, \ldots, \xi'_{n - k})$, we reverse the steps of the above argument to generate the corresponding random tree. 

In remains to estimate the speed of the algorithm. The main step is to generate $ (\xi'_1, \ldots, \xi'_{n - k})$, or, which is the same, a sequence $(\hat \zeta_1, \ldots, \hat \zeta_{n - k})$  Note that for $\hat \zeta_i$ we have $\E S_{\hat \zeta}(n - k)= (n - k)/(1 - \alpha)\sim n$.)

A fast method for generating such sequences was developed by Devroye \cite{devroye2012}. In application to our situation, the method is as follows. 
\begin{enumerate}
\item Generate a multinomial random vector $(N_1, N_2, \ldots, )$ with parameters $(n - k, \wh w_1^\ast, \wh w_2^\ast, \ldots )$. Repeat this step until a sample is obtained such that 
\bal{
\sum_{j = 0}^\infty j N_j = n - 1. 
}
\item Form an array of length $n - k$ with $N_j$ values $j$ and randomly permute it. The result is the desired sequence $\xi'_1, \ldots, \xi'_{n - k}$.  
\end{enumerate}
Let us analyze the running time of this algorithm. 

By Theorem 1 in \cite{devroye2012}, there is a method to generate the required multinomial vector in expected time $\tau_n = o(n^{1/2})$ provided that the sequence $\wh w_i^\ast$ has finite variance. (In fact, for a compactly supported distribution, this time is $O(1)$.)

The distribution of the sum $\sum_{j = 0}^\infty j N_j$ is the same as the distribution of the sum of an i.i.d. sequence $\hat \zeta_1, \ldots, \hat \zeta_{n - k}$ where each r.v. has pmf $\{\wh w_i^\ast\}$.  The expectation of the pmf $\{\wh w_i^\ast\}$ is $1/(1-\alpha)$ by its definition (\ref{defiHatShifted}), hence, the expectation of the sum is $(n - k)/(1 - \alpha) = n$. By using the local central limit theorem (valid under the assumption that the variance of pmf $\wh w_i^\ast$ is finite), the probability that $\sum_{i = 1}^{n - k} \hat \zeta_i = n - 1$ is bounded from below by $\Omega(1/\sqrt{n})$. Therefore, the expected time that we draw a sample $N_j$ such that   $\sum_{j = 0}^\infty j N_j = n - 1$ is $\tau_n \sqrt{n} = o(n)$. 

The time for other operations (adding zeros and random shuffling) is linear in $n$. 
\end{proof}

%

\begin{proof}[Proof of Theorem \ref{theoConvergenceToExcursion}]
The proof consists of several steps. In the first step, we will show that the \L ukasiewicz path of a random tree in $\Omega_{k, n}$ converges to a Brownian excursion.
This will be accomplished in Lemmas \ref{lemmaSecondBridge} and \ref{lemmaLukasToExcursion} by analyzing random sequences produced by Algorithm A. 

In the second step of the proof, we are going to show that the \L ukasiewicz path of a random tree from $\Omega_{k,n}$ is close  to the height process of the tree. This result is formulated in Lemma \ref{lemmaSandH}. 

In the third step, we will show that with high probability the height process is  close to the contour process. See Lemma \ref{lemmaHandC}.

The final step is to bring all these results together and show that the contour process for a tree in $\Omega_{k, n}$ converges to the Brownian excursion. 

\textbf{In all lemmas in this section   (Lemmas \ref{lemma_xi_moments} - \ref{lemmaSandC}), we assume that the assumptions of Theorem \ref{theoConvergenceToExcursion} are satisfied. That is, condition (\ref{basicCondition}) holds, and the sequence of pairs $(k, n)$ is such that $(k, n) \in PP(w)$ and, as $n \to \infty$,  $k =  \alpha n + O(1)$, where $\alpha \in (0, 1 - 1/\wh\nu)$.}

\bigskip
%

\textbf{Step 1: \L ukasiewicz path converges to Brownian excursion.}

Let $\xi''_1, \ldots, \xi''_n$ be the sequence of random variables produced by Step 2 of Algorithm A. Some properties of these random variables are proved in Lemmas \ref{lemma_xi_moments} and \ref{lemma_xi2_sum}. Then, we prove the convergence of \L ukasiewicz path to Brownian excursion in Lemmas \ref{lemmaSecondBridge} and \ref{lemmaLukasToExcursion}.

\begin{lemma} 
\label{lemma_xi_moments}
 As $n\to \infty$, we have:
\begin{enumerate}[(a)]
\item For each integer $s \geq 0$,
$\E((\xi''_1)^s) \to \E((L^\ast)^s)$, where $L_\ast$ is a random variable with distribution $w_j^\ast$.   In particular, $\Var(\xi''_1) \to (\sigma^\ast)^2$;
\item For each integer pair $s, t \geq 0$,  $\Cov\Big((\xi''_1)^s, (\xi''_2)^t\Big) \to 0$.
\item For every $\beta > 0$, there exists a $\gamma > 0$ such that 
\bal{
\P(\max_{j} |\xi''_j |> n^\beta) \leq e^{-\gamma n^\beta}.
}
\end{enumerate}
\end{lemma}

\begin{proof}
(a)
First, consider variables $\xi'_i$. They have the distribution:
\bal{
\P\Big(\xi'_1 = j_1, \ldots, \xi'_{n - k} = j_{n - k}\Big) = \P\Big(\hat\zeta_1 = j_1, \ldots, \hat\zeta_{n - k} = j_{n - k}\Big| S_{\hat \zeta} (n - k) = n - 1\Big), 
}
where $\hat\zeta_1, \ldots, \hat\zeta_{n - k}$ are i.i.d. random variables with distribution $\P(\hat\zeta_1 = j) = \wh w^{\ast}_j = w^{\ast}_j/ (1- \alpha)$, if $j \geq 1$, and $0$ if $j = 0$. Here  $S_{\hat \zeta} (p)$ denotes  $\sum_{i = 1}^p \hat\zeta_i$. 
Therefore, 
\bal{
\P(\xi'_1 = j) &= \P\Big(\hat\zeta_1 = j \Big |  S_{\hat \zeta} (n - k) = n - 1\Big) 
=\frac{\P\Big(\hat\zeta_1 = j, \sum_{i=2}^{n - k} \hat\zeta_i = n - 1 - j\Big ) }
{\P\Big(S_{\hat \zeta} (n - k) = n - 1\Big) } 
\\
&= \P(\hat\zeta_1 = j)\frac{\P\Big(S_{\hat \zeta} (n - k - 1) = n - 1 - j\Big)}
{\P\Big(S_{\hat \zeta} (n - k) = n - 1\Big) }.
}

The distribution of the random variable $\xi''_1$ can be obtained as a mix distribution: $\xi''_1 = 0$ with probability $(k/n)$ and $\xi''_1=\xi'$ with probability $(1 - k/n)$, where $\xi'$ has the same distribution as $\xi'_1$. Also, the distribution $L^\ast$ is a mix of the trivial distribution with probability $\alpha$ and the distribution of $\hat \zeta$ with probability $(1 - \alpha)$. So, for $j > 0$, 

\bal{
\P(\xi''_1 = j) &= (1 - k/n) \P(\xi'_1 = j),
\\
\P(L^\ast = j)& = (1 - \alpha) \P(\hat \zeta_1 = j),
}
and therefore,
\bal{
\P(\xi''_1 = j) =  \P(L^\ast = j) \frac{1 - k/n}{1 - \alpha}\frac{\P\Big(S_{\hat \zeta} (n - k - 1) = n - 1 - j\Big)}
{\P\Big(S_{\hat \zeta} (n - k) = n - 1\Big) }.
}

Next, both $k$ and $n$ grow to infinity in such a way that $k/n \to \alpha\in(0, 1)$, hence $\E S_{\hat \zeta} (n - k) = (n - k) \E\hat \zeta_1 = n (1 - k/n) / (1 - \alpha) = n + o(1)$ and similarly $\E S_{\hat \zeta} (n - k - 1) = n + o(1)$. We apply the local limit law (Theorem I.4.2 in Kolchin \cite{kolchin84}) and find that  for any $\eps > 0$ we can find $n(\eps)$ such that for all $n > n(\eps)$ and all $j < n^{1/4}$, 
\bal{
\bigg|\frac{1 - k/n}{1 - \alpha}\frac{\P\Big(S_{\hat \zeta} (n - k - 1) = n - 1 - j\Big)}
{\P\Big(S_{\hat \zeta} (n - k) = n - 1\Big) } - 1\bigg|< \eps.
}
In addition, this ratio of probabilities is bounded uniformly in $j$ and $n$. 

Then, for arbitrary $\eps > 0$ and $n > n(\eps)$, and for some $c > 0$. 
\bal{
\bigg|\E\Big[ (\xi''_1)^s\Big] -  \E\Big[ (L^\ast)^s\Big] \bigg|
= \eps \sum_{j = 0}^{\lfloor n^{1/4}\rfloor} j^s\, \P(L^\ast = j) + c\sum_{j ={\lfloor n^{1/4}\rfloor} + 1}^{\infty} j^s\, \P(L^\ast = j).
}

By assumption, the probability distribution of $L$ has exponentially declining tails, which implies that $L^\ast$ has exponentially declining tails, so the last term is bounded by
\bal{
c\sum_{j ={\lfloor n^{1/4}\rfloor} + 1}^{\infty} j^s\, e^{-a j} \to 0,
}
as $n\to \infty$, 
and we conclude that all moments of random variable $\xi''_1$ converge to the corresponding moments of $L^\ast$, as $n \to \infty$. In particular,  $\Var(\xi''_1) \to (\sigma^\ast)^2$.


(b) By a similar argument, we can write for $j_1 > 0, j_2 > 0$,
\bal{
\P(\xi''_1 = j_1, \xi''_2 = j_2) &=  \P(L_1^\ast = j_1, L_2^\ast = j_2) \Big(\frac{1 - k/n}{1 - \alpha}\Big)^2
\\
&\times \frac{\P\Big(S_{\hat \zeta} (n - k - 2) = n - 1 - j_1 - j_2\Big)}
{\P\Big(S_{\hat \zeta} (n - k) = n - 1\Big) },
}
where $L_1^\ast$, $L_2^\ast$ are two independent random variables that have the $\alpha$-shifted distribution. By using the local limit law, we find that the ratio of probabilities approaches $1$ as $n \to \infty$, uniformly for $j_1 < n^{1/4}, j_2 < n^{1/4}$. In addition, this ratio is bounded uniformly in $n, j_1, j_2$. Together with assumption (1), this ensures that  $\E\Big((\xi''_1)^s (\xi''_2)^t\Big) \to \E\Big((L^\ast_1)^s (L^\ast_2)^t\Big)$ as $n \to \infty$, for all integer $s\geq 0, t\geq 0$. In particular, this implies that $\Cov\Big((\xi''_1)^s, (\xi''_2)^t\Big) \to 0$, as $n \to \infty$.

(c)  We have 
\bal{
\P&(\max_i\xi'_i \leq n^\beta) = \P\Big(\xi'_{1} \leq n^\beta, \ldots, \xi'_{n - k}\leq n^\beta\Big)  
\\
&= \P\Big(\hat\zeta_{1} \leq n^\beta, \ldots, \hat\zeta_{n - k}\leq n^\beta \Big| S_{\hat\zeta}(n - k) = n - 1\Big) 
\\ 
&=\Big[\P(\hat\zeta_{1} \leq n^\beta)\Big]^{n - k} 
\frac{\P\Big(S_{\hat\zeta}(n - k) = n - 1)\Big| \hat\zeta_{1} \leq n^\beta, \ldots, \hat\zeta_{n - k}\leq n^\beta\Big) }{\P(S_{\hat\zeta}(n - k) = n - 1)}.
}
Define truncated random variables $\hat\zeta_i^{(n)}$, $i = 1, \ldots, n - k$ as i.i.d random variables with distribution 
\bal{
\P(\hat\zeta_1^{(n)} = j) = \P\Big(\hat\zeta_1 = j \Big| \hat\zeta_1 \leq n^\beta\Big).
} 
Then we can write:
\begin{equation}
\label{estimate_on_max}
\P(\max_i\xi'_i \leq n^\beta) = \Big[\P(\hat\zeta_{1} \leq n^\beta)\Big]^{n - k} 
\frac{\P\Big(S_{\hat\zeta^{(n)}}(n - k) = n - 1)\Big)}{\P\Big(S_{\hat\zeta}(n - k) = n - 1\Big)}.
\end{equation}
By the assumption on the offspring probabilities, $\wh\zeta_1$ has exponentially declining tails and therefore, for some $a > 0$, we can write:
\bal{
(n - k)\log \P(\hat\zeta_{1} \leq n^\beta) \geq (n - k) \log (1 - e^{-a n^\beta}) \geq -2 n e^{-a n^\beta},
}
where the last inequality follows because by a choice of $a$ we can ensure $e^{-a n^\beta} < 3/4$ and we have $\log(1 - x) \geq -2x$ for $0 \leq x \leq 3/4$. Then,  
\bal{
\Big[\P(\hat\zeta_{1} \leq n^\beta)\Big]^{n - k} &\geq e^{-2n e^{-a n^\beta}} \geq 1 -2n e^{-a n^\beta}.
}
This implies that for sufficiently small $\gamma> 0$, 
 $\big[\P(\hat\zeta_{1} \leq n^\beta)\big]^{n - k} \geq 1 - e^{-\gamma n^\beta}$. 
 
 To estimate the denominator in the ratio of probabilities in (\ref{estimate_on_max}), we can use the local limit theorem (Theorem I.4.2 in Kolchin \cite{kolchin84}). For the numerator, this theorem is not applicable directly, since the distribution of random variables $\hat\zeta_i^{(n)}$ is changing with $n$. However, it can be adapted to this more general setting along the lines in the proof of Theorem I.4.2 in Kolchin. (We provide the proof of the modified local limit theorem in ``Appendix''.)
 Then, we find that, for $r > 0$, 
\bal{
\frac{\P\Big(S_{\hat\zeta^{(n)}}(n - k) = n - 1)\Big)}{\P\Big(S_{\hat\zeta}(n - k) = n - 1\Big)} \to 1
}
as $n \to \infty$. It follows that for some smaller $\gamma>0$, $\P(\max_i\xi'_i \leq n^\beta) \geq 1 - e^{-\gamma n^\beta}$. In a similar fashion we can show that $\P(\min_i \xi'_i \geq - n^\beta) \geq 1 - e^{-\gamma n^\beta}$. The sequence $\xi''_{(1)}, \ldots, \xi''_{(n)}$ is different from the sequence $\xi'_{(1)}, \ldots, \xi'_{(n-k)}$ only by addition of $k$ zeros. It follows that some $\gamma > 0$,
$\P(\max_{i} |\xi''_i | \leq n^\beta) \geq 1 - e^{-\gamma n^\beta}$, which implies the claim of the lemma.  

\end{proof}

\begin{lemma} 
\label{lemma_xi2_sum}
As $n \to \infty$,
\bal{
\frac{1}{n}\sum_{i = 1}^n \Big(\xi''_i -  1 + \frac{1}{n}\Big)^2 \pcv (\sigma^\ast)^2.
}
\end{lemma}
\begin{proof}
We have $\E\xi''_i = 1 - 1/n$ and  $\E\big[ (\xi''_i -  1 + \frac{1}{n})^2\big] \to (\sigma^\ast)^2$ by Lemma \ref{lemma_xi_moments}(a).
Then, 
\bal{
\Var&\Big(\frac{1}{n}\sum_{i = 1}^n \big(\xi''_i -  1 + \frac{1}{n}\big)^2\Big) 
= \frac{1}{n^2} \Big[\sum_{i = 1}^n\Var \Big(\big(\xi''_i -  1 + \frac{1}{n}\big)^2\Big) 
\\
&+ 2 \sum_{i < j}\Cov \Big(\big(\xi''_i -  1 + \frac{1}{n}\big)^2, \big(\xi''_j -  1 + \frac{1}{n}\big)^2\Big)\Big]
\\
& = \frac{1}{n^2} \Big[n \Var \Big(\big(\xi''_1 -  1 + \frac{1}{n}\big)^2\Big) + n(n - 1) 
\Cov \Big(\big(\xi''_1 -  1 + \frac{1}{n}\big)^2, \big(\xi''_2 -  1 + \frac{1}{n}\big)^2\Big)
\Big]
\\& \to 0,
}
by Lemma \ref{lemma_xi_moments}(a) and (b). By Chebyshev's inequality, this implies the statement of the lemma.
\end{proof}

  Let $ S''_n(0) = 0$ and 
\bal{
S''_n(j) = \sum_{i = 1}^j \xi''_i, \text{ where } j = 1, \ldots, n,
}
and extend $S''_n(t)$ to all $t \in [0, n]$ by linear interpolation. 

In the lemmas below, the convergence of stochastic processes is the weak convergence with respect to the space of continuous functions on $[0, 1]$ with the topology induced by the uniform norm: $\|f\|_{\infty} = \sup_{t \in [0,1]} f(t)$.

\begin{lemma}
\label{lemmaSecondBridge}
As $n \to \infty$
\bal{
\frac{S''_n(tn) - tn}{\sigma^\ast \sqrt{n}}\Big|_{t \in [0,1]} \xrightarrow{weakly}  b(t)\Big|_{t \in [0,1]}, 
}
 where $b(t)$ is the standard Brownian bridge on the interval $[0, 1]$, and $\sigma^\ast$ is as defined in (\ref{variance}). 
\end{lemma}
\begin{proof}
Define 
$x_j = (\xi''_j - 1 + 1/n)/(\sqrt{n} \sigma^\ast)$
 and $X_n(t) = \sum_{j = 1}^{[tn]} x_j$, with $X_n(t) = 0$ for $0 \leq t < 1/n$. 
  Then, $\sum_{j = 1}^{n} x_j= 0$ and $\sum_{j = 1}^{n} x_j^2 \pcv  1$ by Lemma \ref{lemma_xi2_sum}. In addition, by Lemma \ref{lemma_xi_moments}(c), $\max_{j} |x_j| \pcv 0$. This means that the assumptions of Theorem 24.2 in Billingsley's book \cite{billingsley1968} are satisfied, and we have $X_n(t) \xrightarrow{weakly}  b(t)$. This implies the claim of the lemma. 
\end{proof}

Next, we consider the application of Step 3 in Algorithm A. By Theorem \ref{theoAlgorithm}, this step produces the \L ukasiewicz path of a random tree in $\Omega_{k, n}$. 

\begin{lemma}
\label{lemmaLukasToExcursion} 
As $n \to \infty$, the normalized \L ukasiewicz path $S_n(tn)/(\sigma^\ast \sqrt{n})$ of a random tree from $\Omega_{k, n}$ weakly converges to a standard Brownian excursion $e(t)$. 
\end{lemma}

\begin{proof}
Let $b''_n(t) = (S''_n(tn) - tn)/(\sigma^\ast \sqrt{n})$. By the previous lemma, $b''_n(t)$ converges to $b(t)$ in distribution. According to Vervaat's Theorem (Theorem 1 in \cite{vervaat79}), the location of the absolute minimum of Brownian bridge $b(t)$ is unique with probability $1$, and, if this time is denoted $\tau_m$, the process $b(t + \tau_m) - b(\tau_m)$ is distributed as a Brownian excursion $e(t)$. (Here, the addition in the argument is modulo $1$.)

Since $b''_n(t)$ converges to the Brownian bridge $b(t)$ in distribution, we can couple the random variables $ b''_n(t)$ and  $b(t)$, so that the time at which $b''_n(t)$ achieves its first minimum, $\tau''_m$, converges in probability to $\tau_m$. Then, 
\bal{
b''_n(t + \tau''_m) - b''_n(\tau''_m)\Big|_{t \in [0, 1]} \xrightarrow{dist} b(t + \tau_m) - b(\tau_m)\Big|_{t \in [0, 1]},
}
which is distributed as a Brownian excursion by Vervaat's theorem. 

If we write $S_n(j) = \sum_{i = 1}^{j}(\wh\xi''_i - 1)$, where $(\wh\xi''_i)$ is the sequence $(\xi''_i)$ after the cyclic shift, then the process $b''_n(t + \tau''_m) - b''_n(\tau''_m)$ is equal to 
$S_n(t)/(\sigma^\ast \sqrt{n})$. By Theorem \ref{theoAlgorithm}, $S_n(t)$ is distributed as the \L ukasiewicz path of a random tree from $\Omega_{k, n}$. 

Hence, we have shown that the normalized \L ukasiewicz path $S_n(tn)/(\sigma^\ast \sqrt{n})$ of a random tree from $\Omega_{k, n}$ converges to a standard Brownian excursion $e(t)$. 
\end{proof}

%

\textbf{Step 2: \L ukasiewicz path is close to the height process.} 

The height process for an ordered tree is defined in (\ref{height_process}). In Lemma 1 of Marckert and Mokkadem\cite{marckert_mokkadem2003}, it was shown that if $S_n(j)$ is the \L ukasiewicz path of an ordered tree, then the height process $H_n(l)$, $l = 0, \ldots, n - 1$,  can be written as follows:
\begin{equation}
\label{rightMinima0}
H_n(l) = \#\Big\{i : 0\leq i \leq l - 1, S_n(i) = \min_{i \leq k \leq l} \big\{ S_n(k)\big\}\Big\} 
\end{equation}
In words, $H_n(l)$ is the number of (weak) right minima of $S_n$ on the interval $[0, l]$ (not counting the trivial right minimum at $i = l$).

For example, in Fig. \ref{figTree}, $H_8(3) = 2$, and this corresponds to the right minima at $i = 0$ and $i = 2$ of the \L ukasiewicz path $S_8(i)$ on interval $[0, 3]$. 

Now, let $W_n(j)$, $j = 0, \ldots, n$,  be a path of a random walk with independent increments $x_i$ that can take values $s = -1, 0, 1, \ldots$ with the following probabilities:
\begin{equation}
\label{stepDistribution}
\P(x_i = s) = w^\ast_{s + 1},   
\end{equation} 
where 
$\{w^\ast_k\}_{k \geq 0}$ is the $\alpha$-shifted weight distribution from Definition \ref{defiAlphaShift} . For a path of this random walk, we can also define the number of weak right minima on the interval $[0, l]$,
$ l = 0, \ldots, n - 1$: 
\begin{equation}
\label{rightMinima1}
RM_n(l) = \#\Big\{i : 0\leq i \leq l - 1, W_n(i) = \min_{i \leq k \leq l} \big\{ W_n(k)\big\}\Big\} 
\end{equation}

The key result about the process $RM_n(l)$ was proved in \cite{marckert_mokkadem2003}.
\begin{lemma}[Corollary 5 in \cite{marckert_mokkadem2003}]
\label{lemma_RM}
 Let $W_n(j)$, $j = 0, \ldots, n$,  be a path of a random walk with independent increments that have the law as in (\ref{stepDistribution}), and let $RM_n(l)$ be as defined in (\ref{rightMinima1}). Then, for some $\gamma > 0$ and $\nu > 0$, we have
\begin{equation} 
\label{RM_uncond}
\P\bigg(\sup_{0 \leq l < n} \Big| W_n(l) -  \min_{0 \leq i \leq l} W_n(i)
- RM_n(l) \frac{(\sigma^\ast)^2}{2}\Big| \geq n^{1/4 + \nu} \bigg) \leq e^{-\gamma n^\nu}.
\end{equation}
\end{lemma}

Note that if $W_n(j)$ happens to be an excursion, then $\min_{0\leq i \leq l} W_n(i) = 0$ and in this case Lemma \ref{lemma_RM} says that $RM_n(l)$ after scaling by $(\sigma^\ast)^2/2$ is close to the excursion $W_n(l)$. 

The proof of this result in \cite{marckert_mokkadem2003} is based on a reduction of the problem to a problem about ladder heights of a random walk, for which appropriate tools are available in Feller's book (\cite{feller71}, Chapter XII).

By using this result, we can complete the second step of the proof. 
\begin{lemma}
\label{lemmaSandH}
 Let $S_n(l)$ and $H_n(l)$ be the \L ukasiewicz and the height processes for a random tree in $\Omega_{k, n}$. 
For any $\nu > 0$, there exists constants $\gamma > 0$ and $N > 0$ such that for all $n \geq N$ and $(n, k) \in PP(w)$, 
\begin{equation} 
\label{RM_cond}
\P\bigg(\sup_{0 \leq l < n} \Big| S_n(l) - 
 \frac{(\sigma^\ast)^2}{2} H_n(l)\Big| \geq n^{1/4 + \nu} \bigg) \leq e^{-\gamma n^\nu}.
\end{equation}
\end{lemma}
\begin{proof}
From (\ref{rightMinima0}) and (\ref{rightMinima1}), it can be seen that the probability of the event in (\ref{RM_cond})  equals the probability of the event in (\ref{RM_uncond}), denoted as $\mathfrak{A}_n$, conditional on the event $\mathfrak{B}_n$ that $W_n(i)$, $i = 0, \ldots, n$ is an excursion with exactly $k$ steps equal to $-1$. We have 
\bal{
\P(\mathfrak A_n | \mathfrak B_n) \leq \frac{\P(\mathfrak A_n)}{\P(\mathfrak B_n)} \leq \frac{e^{-\gamma n^\nu}}{\P(\mathfrak B_n)}.
}
It follows that it is enough to show that $\P(\mathfrak B_n)$ is sufficiently large. Let us introduce two other events: $\mathfrak C_n$ is the event that $W_n(i)$ has exactly $k$ steps equal to $-1$ and $\mathfrak{E}_n$ is the event that $W_n(i)$ is an excursion. Then $\mathfrak B_n = \mathfrak C_n \cap \mathfrak E_n$ and 
\bal{
\P(\mathfrak B_n) = \P(\mathfrak E_n | \mathfrak C_n) \P(\mathfrak C_n). 
}

The probability of the increment $-1$ in the random walk $W_n$ equals to $w^\ast_0 = \alpha$, so by the local central limit theorem the probability of the event $\mathfrak C_n$ is asymptotically equal to $a/\sqrt{n}$ with $a > 0$. 

In order to calculate $\P(\mathfrak E_n | \mathfrak C_n)$ we note that the distribution of the increments conditional that they are non-negative is given  by 
\bal{
\P(x_j = s | x_j \geq 0) = \frac{w^\ast_{s + 1}}{1 - \alpha}, \text{ for } s = 0, 1, 2, \ldots 
}
The conditional expectation of the increment equals 
\bal{
\E(x_j | x_j \geq 0) &= \frac{1}{1 - \alpha}\sum_{s = 0}^\infty s w^\ast_{s+1} = \frac{1}{1 - \alpha}\sum_{k = 1}^\infty (k - 1) w^\ast_{k} 
\\
&= \frac{1}{1 - \alpha}\Big(\sum_{k = 1}^\infty k w^\ast_{k} - \sum_{k = 1}^\infty  w^\ast_{k}\Big) = \frac{1}{1 - \alpha}\Big(1 - (1 - \alpha)\Big)
\\
& = \frac{\alpha}{1 - \alpha}. 
}
Then the expectation of the sum $W_n(n)$ conditional that exactly $k$ of the increments are equal to $-1$ is
\bal{
\E( W_n(n) | \mathfrak C_n) &= - k + \E(x_j | x_j \geq 0) (n - k) 
\\
&= - \alpha n + \frac{\alpha}{1 - \alpha} (1 - \alpha) n + O(1)
= O(1). 
}
Again, by using the local central limit theorem, we find that the probability that $W_n$ is a random walk bridge, conditional on the event $\mathfrak C_n$, is asymptotically equal to $b/\sqrt{n}$ with $b > 0$. 
Then, all $n$ cyclic shifts of this random walk bridge have the same probability by exchangeability of the increments and exactly one of these shifts is an excursion by the cyclic lemma. Hence, conditional on $\mathfrak C_n$, the random walk $W_n$ is an excursion with probability $\sim b/n^{3/2}$. Altogether we conclude that 
\bal{
\P(\mathfrak B_n) \sim \frac{a b}{n^{5/2}}, 
} 
with $a b > 0$. 
Hence, 
\bal{
\P(\mathfrak A_n | \mathfrak B_n) \leq \frac{1}{ab} n^{5/2} e^{-\gamma n^\nu} \leq e^{-\gamma' n^\nu},
}
for a suitably chosen $\gamma'$ and this completes the proof of the lemma. 
\end{proof}

%

\textbf{Step 3. The height process is close to the contour process.}

Here we introduce a couple of purely combinatorial results from Marckert and Mokkadem\cite{marckert_mokkadem2003}  that relate the height and contour processes. 

 \begin{defi} 
 \label{defi_ml}
Let $T\in \Omega_n$ be a rooted planar tree on $n$ vertices. For every $l = 0, \ldots, n - 1$, define $m(l)$ as the first time when the Depth First Search reaches the vertex $v_l$.
 \end{defi}
 
  It is clear that the sequence $\Big(m(l)\Big)_{l = 0}^{n -1}$ is strictly  increasing. 
 
 In the example in Fig. \ref{figTree}, we have:
 \begin{center}
 \begin{tabular}{c | c c c c c c c c}
 l & 0 & 1 & 2 & 3 & 4 & 5 & 6 & 7 \\
 m(l) & 0 & 1 & 3 & 4 & 7 & 8 & 9 & 12 
 \end{tabular} 
 \end{center}

 By definition of the contour process $C_n(t)$, it is easy to see that for all $l$, the points $\big(m(l), H_n(l)\big)$ are located on the graph of the contour process.  Note also that for all $k \in \big[m(l), m(l + 1)\big)$, 
 \begin{equation}
 \label{interpolationBounds}
 H_n(l + 1) - 1 \leq C_n(k) \leq H_n(l).
\end{equation} 
 (see Lemma 3 in \cite{marckert_mokkadem2003}). 
 Intuitively, $\big(m(l), H_n(l)\big)$ interpolate the contour process, and $l \mapsto m(l)$ can be thought as a time change that connects the contour process to the height process. See the illustration in Fig. \ref{figTree}.

 Another important combinatorial result proved in \cite{marckert_mokkadem2003} is that for every $l \in \{0, \ldots, n - 1\}$,
 \begin{equation}
 \label{mIdentity}
 m(l) = 2l - H_n(l).
 \end{equation} 
 Since for large $l$, we can expect that $H_n(l)$ has the order $\sqrt{l}$ with large probability, this identity suggests that the time change function $m(l) \sim 2 l$ with large probability, and that the deviations of $m(l)$ from $2l$ are of order $\sqrt{l}$. 
 
 Now, let us define $\tau(nt)$ as the integer such that 
 \begin{equation}
 \label{defiZeta}
 \lfloor 2n t \rfloor \in \Big[ m\big( \tau(nt) \big), m\big( \tau(nt) + 1 \big)\Big). 
\end{equation}
 (Intuitively, this function plays the role of the inverse of $m(l)$ scaled by $2$.) 
 
 Then we have the following result. 
 \begin{lemma}
 \label{lemmaHandC}
 For all $a>0$ and $\nu > 0$, there exists such $\gamma > 0$ that 
 \bal{
 \P\Big( \sup_{t \in [0, 1]} \Big| H_n\big(\tau(nt)\big)
 - C_n(2nt) \Big| \geq a n^{1/4 +\nu}\Big) \leq e^{-\gamma n^\nu}
 }
 \end{lemma} 
 \begin{proof}
 By (\ref{interpolationBounds}), we have 
  \bal{
 \P\Big(\sup_{t \in [0, 1]} \Big| H_n\big(\tau(nt)\big)
 - C_n(2nt) \Big| \geq a n^{1/4 +\nu}\Big) 
 \\
 \leq \P\Big( \sup_{l \in \{0, \ldots, n-2\}} \Big| H_n\big(l\big)
 - H_n\big(l + 1\big) + 1\Big| \geq a n^{1/4 +\nu}\Big)
 \\
 \leq \P\Big(\sup_{l \in \{0, \ldots, n-2\}} \Big| H_n\big(l\big)
 - \frac{2}{(\sigma^\ast)^2}S_n(l)\Big| \geq \frac{a}{3}  n^{1/4 +\nu}\Big)
\\
 + \P\Big(\sup_{l \in \{0, \ldots, n-2\}} \Big|  H_n\big(l + 1\big) - 1
 -  \frac{2}{(\sigma^\ast)^2}S_n(l + 1)\Big| \geq \frac{a}{3} n^{1/4 +\nu}\Big),
 \\
 + \P\Big(\sup_{l \in \{0, \ldots, n-2\}} \frac{2}{(\sigma^\ast)^2} \Big|  S_n\big(l + 1\big) 
 -  S_n(l)\Big| \geq \frac{a}{3}  n^{1/4 +\nu}\Big),
 }
 and by Lemmas \ref{lemma_xi_moments}(c) and \ref{lemmaSandH}, the sum of these three probabilities can be bounded by $e^{-\gamma n^\nu}$ with a suitable $\gamma > 0$. 
 \end{proof}

 By using Lemmas \ref{lemmaSandH} and \ref{lemmaHandC}, we can relate the \L ukasiewicz process $S(\tau(nt))$ and the contour process $C_n(nt)$. To move further, we need to estimate the difference between $\tau(nt)$ and $nt$. 
 
 \begin{lemma} 
 \label{lemmaZeta}
 Let $\tau(nt)$ be as in (\ref{defiZeta}) for the random time $m(l)$ defined for a random tree in $\Omega_{k,n}$. Then, 
 \bal{
 \P\Big(\sup_{t \in [0, 1]} \big|\tau(nt) - nt\big| > n^{1/2 + \nu}\Big) \leq e^{-\gamma n^\nu}.
 }
 \end{lemma}
 \begin{proof}
 By the triangle inequality, we have 
 \bal{
 \sup_{t \in [0, 1]} \big|2\tau(nt) - 2nt\big| \leq \sup_{t \in [0, 1]} \big|2\tau(nt) - m(\tau(nt))\big|
 + \sup_{t \in [0, 1]} \big|m(\tau(nt)) - 2nt\big|.
 }
 By applying Equation (\ref{mIdentity}) with $l = \tau(nt)$, we find that the first term  is bounded by $\sup_t H_n\big(\tau(nt)\big)$. For the second term, we note that by definition of $\tau(nt)$, it is bounded by 
 \bal{
 \sup_{t \in [0, 1]} \Big|m\big(\tau(nt) + 1\big) - m\big(\tau(nt)\big)\Big|
 &=  \sup_{l \in \{0, \ldots, n - 1\}} \big|m(l + 1) - m( l )\big|
 \\
 &\leq \sup_{l \in \{0, \ldots, n - 1\}} H_n(l). 
 }
 Altogether, we have
 \bal{
  \sup_{t \in [0, 1]} \big|2\tau(nt) - 2nt\big| \leq 2 \sup_{l \in \{0, \ldots, n - 1\}} H_n(l).
 }

By (\ref{rightMinima0}), $H_n(l)$ is the number of right minima for $S_n$ on the interval $[0, l]$. It was shown in Lemma 5 in \cite{marckert_mokkadem2003}, that the number of right minima for the random walk $W_n$ satisfies the bound 
\bal{
\P\Big(\sup_{l \in \{0, \ldots, n - 1\}} RM_n(l) \geq n^{1/2 + \nu}\Big) \leq e^{-\gamma n^\nu}
}
for a $\gamma > 0$.  An inequality of this form also holds for the number of right minima of $S_n$ by a conditioning argument similar to the argument that we used in the proof of Lemma \ref{lemmaSandH}. 

So we have, for some $\gamma > 0$, 
\bal{
\P\Big(\sup_{t \in [0, 1]} \big|\tau(nt) - nt\big| \leq n^{1/2 + \nu}\Big) 
\geq \P\Big(2 \sup_{l\in\{0, \ldots, n-1\}}  H_n(l) \leq n^{1/2 + \nu}\Big) \geq 1 - e^{-\gamma n^\nu}
}
and 
\bal{
\P\Big(\sup_{t \in [0, 1]} \big|\tau(nt) - nt\big| > n^{1/2 + \nu}\Big) 
 \leq e^{-\gamma n^\nu}
}

 \end{proof}

\begin{lemma}
\label{lemmaSandC}
Let $S_n(t)$ and  $C_n(t)$ be the \L ukasiewicz and the contour processes for a tree $T \in \Omega_{k, n}$. Then, for every $\nu > 0$, there exist two positive constants, $\gamma > 0$ and $N > 0$, such that for all $n \geq N$, 
\bal{
\P\bigg[ \sup_{t \in [0,1]} \Big| S_n(nt) - \frac{\sigma^2}{2}C_n(2nt)\Big|
\geq n^{1/4 + \nu}\bigg] \leq e^{-\gamma n^\nu}.
}
\end{lemma}

\begin{proof}
We have 
\bal{
\P\Big(\sup_{t \in [0, 1]} \big| S_n(nt) - \frac{\sigma^2}{2}C_n(2nt)\big| \geq n^{1/4 + \nu} \Big) \leq A + B +C + D, 
}
where 
\bal{
A = \P\bigg(\sup_{t \in [0, 1]} \Big| S_n(nt) - S_n\big(\tau(nt)\big)\Big| \geq \frac{1}{3} n^{1/4 + \nu},
\sup_t\big|\tau(nt) - nt \big| < n^{1/2 + \nu}\bigg),
}
\bal{
B = \P\bigg(\sup_{t \in [0, 1]} \Big| S_n\big(\tau(nt)\big) - \frac{\sigma^2}{2}  H_n\big(\tau(nt)\big)\Big| \geq \frac{1}{3} n^{1/4 + \nu}\bigg),
}
\bal{
C = \P\bigg(\frac{\sigma^2}{2}\sup_{t \in [0, 1]} \Big| H_n\big(\tau(nt)\big)- C_n (2nt)\Big| \geq \frac{1}{3} n^{1/4 + \nu}\bigg),
}
\bal{
D = \P\Big( \sup_{t \in [0, 1]} \big|\tau(nt) - nt\big| > n^{1/2 + \nu}\Big).
}
The quantities $B$, $C$, and $D$ can be estimated using Lemmas \ref{lemmaSandH}, \ref{lemmaHandC}, and \ref{lemmaZeta}, respectively. For quantity $A$, note that we can estimate a similar quantity for unconditional random walk $W_n$:

\bal{
\P\bigg(\sup_{t \in [0, 1]} \Big| W_n(nt) - W_n\big(\tau(nt)\big)\Big| \geq \frac{1}{3} n^{1/4 + \nu},
\sup_t\big|\tau(nt) - nt \big| < n^{1/2 + \nu}\bigg)
\\
\leq \P\Big(\sup_{k \leq n^{1/2 + \nu}} \sup_{0 \leq j \leq n - k} \big| W_n(j) - W_n(j + k)\big| \geq \frac{1}{3} n^{1/4 + \nu}\Big)
\\
\leq 
\sum_{j = 1}^n \P \Big(\sup_{k \leq n^{1/2 + \nu}} \big| W_n(k)\big| \geq \frac{1}{3} n^{1/4 + \nu}\Big)
\leq e^{-\gamma n^\nu}.
}
The last step holds by Theorem III.15 from Petrov's book \cite{petrov75} (see also Appendix in \cite{marckert_mokkadem2003}), and here we use the assumption that $\E e^{a L} < \infty$ for some $a > 0$. 

This estimate remains valid for the process $S_n$ instead of $W_n$, perhaps with a different $\gamma > 0$, by a conditioning argument similar to the argument that was used in the proof of Lemma \ref{lemmaSandH}. Therefore, we can find such $\gamma > 0$, that $A + B + C + D \leq e^{-\gamma n^\nu}$.
\end{proof}

Now we can conclude the proof of Theorem \ref{theoConvergenceToExcursion}. In Lemma \ref{lemmaLukasToExcursion} we proved that  the normalized \L ukasiewicz path $S_n(tn)/(\sigma^\ast \sqrt{n})$ of a random tree from $\Omega_{k, n}$ converges to a standard Brownian excursion $e(t)$. Then Lemma \ref{lemmaSandC} implies that the process $\sigma^\ast C_n(2tn)/(2 \sqrt{n})$ also converges to the excursion $e(t)$. This concludes the proof. 

\end{proof}

%

For the proof of Theorem \ref{theoMostLikelyType}, we need a lemma, which allows us to compare the probabilities of various events for unconditional random walks and corresponding random bridges. As in the proof of Theorem \ref{theoAlgorithm}, let $\hat\zeta_1,  \ldots, \hat \zeta_{n-k}$ be an i.i.d sequence with $\hat \zeta_1$ having distribution $\wh w_0^\ast = 0, \hat w_i^\ast = w_i^\ast/(1 - \alpha)$, $i = 1, \ldots$.  Let 
\bal{
m := {\lfloor (n-k)/2 \rfloor},
}
and let $\FF_n^{(1/2)}$ denote the sigma-algebra generated by random variables $\hat\zeta_1,  \ldots, \hat \zeta_{m}$.

\begin{lemma}
\label{lemma_uncond2cond}
Suppose $(n, k_n) \in PP(w)$, $k_n = \alpha n + O(1)$, and the distribution $\{w_i\}$ has finite variance. For every event $\mathfrak{B} \in \FF_n^{(1/2)}$, 
\bal{
\P(\mathfrak{B} | \hat\zeta_1 +   \ldots + \hat \zeta_{n-k} = n - 1) \leq c \P(\mathfrak{B}),
 }
 where $c$ is a positive constant that depends on $\{w_i\}$ but not on $n$.
\end{lemma}

\begin{proof}
For the joint probability, we have 
\bal{
\P(\mathfrak{B} , S_{\hat\zeta}(n - k) = n - 1) &= \sum_{x = 0}^\infty \P\Big(\mathfrak{B} , S_{\hat\zeta}(m) =x\Big) \\
&\times \P\Big(S_{\hat\zeta}(n - k - m) = n - 1 - x)\Big) 
}
Hence, for the conditional probability, 
\bal{
\P(\mathfrak{B} | S_{\hat\zeta}(n - k) = n - 1) &= \sum_{x = 0}^\infty \P\Big(\mathfrak{B} , S_{\hat\zeta}(m) =x\Big) 
\\
&\times \frac{\P\Big(S_{\hat\zeta}(n - k - m) = n - 1 - x)\Big) }{\P(S_{\hat\zeta}(n - k) = n - 1)}.
}
By using the local limit law (Theorem I.4.2 in Kolchin \cite{kolchin84}) we have $\P(S_{\hat\zeta}(n - k) = n - 1) \geq c_1/ \sqrt{n}$ for some $c_1 >0$. By Theorem 2.21 on p.76 in Petrov \cite{petrov95}, $\sup_x\P\Big(S_{\hat\zeta}(n - k - m) = n - 1 - x)\Big) \leq c_2/\sqrt{n}$ for some $c_2 > 0$. Altogether, we get
\bal{
\P(\mathfrak{B} | S_{\hat\zeta}(n - k) = n - 1) \leq \frac{c_2}{c_1} \sum_{x = 0}^\infty \P\Big(\mathfrak{B} , S_{\hat\zeta}(m) =x\Big) = \frac{c_2}{c_1} \P(\mathfrak{B}).
}
\end{proof}

%

\begin{proof}[Proof of Theorem \ref{theoMostLikelyType}]
Let 
\bal{
n_{\xi}(i) &= \Big|\{j: \xi_j = i, 1 \leq j \leq n \}\Big|,\\
n_{\xi'}(i) &= \Big|\{j: \xi'_j = i, 1 \leq j \leq n - k \}\Big|,\\
n_{\hat\zeta}(i) &= \Big|\{j: \hat\zeta_j = i, 1 \leq j \leq n - k \}\Big|,
}
where sequences $(\xi_j), (\xi'_j)$ and $(\hat\zeta_j)$ are as defined in the definition of Algorithm A and in the proof of Theorem \ref{theoAlgorithm}. By Theorem \ref{theoAlgorithm}, the random variable $n_i(T)$ has the same distribution as $n_{\xi}(i)$. For $i = 0$, we have $n_{\xi}(0) = k$ and $k/(n - 1) \to \alpha$ by assumption, so $n_0(T)/(n - 1) \to \alpha = w_0^\ast$. For $i > 0$, we can see from the algorithm that it is enough to show that $n_{\xi'}(i)/(n - k) \to \wh w_i^\ast = w_i^\ast/(1 - \alpha)$ almost surely. 

Let us define additionally, 
\bal{
n_{\hat\zeta}^{(1/2)}(i) &= \Big|\Big\{j: \hat\zeta_j = i, 1 \leq j \leq \frac{n - k}{2} \Big\}\Big|,\\
n_{\xi'}^{(1/2)}(i) &= \Big|\Big\{j: \xi'_j = i, 1 \leq j \leq \frac{n - k}{2} \Big\}\Big|.
}
As before, for shortness of formulas, we define $m := \lfloor\frac{n - k}{2}\rfloor$.

Then, 
\bal{
\frac{n_{\hat\zeta}^{(1/2)}(i)}{m} = \frac{1}{m} \sum_{j = 1}^m \1[\hat\zeta_j = i],
}
which implies $\E\Big(\frac{n_{\hat\zeta}^{(1/2)}(i)}{m}\Big) = \wh w_i^\ast$. Then, by using the independence of random variables $\hat\zeta_j$,
\bal{
\E\Big|\frac{n_{\hat\zeta}^{(1/2)}(i)}{m} - \wh w_i^\ast\Big|^4 
= \frac{1}{m^4}\E \Big(\sum_{j = 1}^m (\1[\hat\zeta_j = i] - \wh w_j^\ast)\Big)^4 = O(m^{-2}). 
}
Then, by Markov's inequality:
\bal{
\P\bigg[\Big|\frac{n_{\hat\zeta}^{(1/2)}(i)}{m} - \wh w_i^\ast\Big| > \eps\bigg] = O\Big(\frac{1}{\eps^4 m^2}\Big)
}

By using Lemma \ref{lemma_uncond2cond}, we infer that this estimate holds also for the conditioned variables $\xi'_j$:
\begin{equation}
\label{equ_A}
\P\bigg[\Big|\frac{n_{\xi'}^{(1/2)}(i)}{m} - \wh w_i^\ast\Big| > \eps\bigg] = O\Big(\frac{1}{\eps^4 m^2}\Big).
\end{equation}
Let
\bal{
n_{\xi'}^{(2/2)}(i) = \Big|\Big\{j: \xi'_j = i, m < j \leq n - k\Big\}\Big|.
}
Since $\{\xi'_j\}_{j = 1}^n$ is an exchangeable sequence, we have 
\begin{equation}
\label{equ_B}
\P\bigg[\Big|\frac{n_{\xi'}^{(2/2)}(i)}{m} - \wh w_i^\ast\Big| > \eps\bigg] = O\Big(\frac{1}{\eps^4 m^2}\Big).
\end{equation}
Together, Eqs. (\ref{equ_A}) and (\ref{equ_B}) imply that 
\begin{equation}
\P\bigg[\Big|\frac{n_{\xi'}(i)}{n - k} - \wh w_i^\ast\Big| > \eps\bigg] = O\Big(\frac{1}{\eps^4 ((n - k)^2}\Big).
\end{equation}
So, by Borel--Cantelli lemma, as $n \to \infty$,
\bal{
\frac{n_{\xi'}(i)}{n - k} \to \wh w_i^\ast,
}
almost surely. This implies the statement of the theorem. 

\end{proof}

\section{Conclusion and Remaining Questions}
\label{section_conclusion}
In this paper, we investigated the scaling limit of Galton--Watson trees conditioned on a fixed number of vertices $n$ and leaves $k$. We considered the case when 
both $k$ and $n$ grow to infinity and $k = \alpha n + O(1)$, with $\alpha \in (0, 1 - 1/\hat\nu)$.  The offspring distribution is assumed to have the exponentially declining tails ($\E e^{a L} < \infty$ for some $a > 0$). 

In this case, we found that the limit under rescaling is the Aldous continuum random tree, and calculated the relevant parameter. We also showed that the height of a tree in $\Omega_{k, n}$ converges under rescaling to the distribution of the maximum of the Brownian excursion. Finally, we showed that the normalized degree sequence of the tree converges almost surely to the $\alpha$-shifted distribution. 

Some other questions remain open. First, what happens when the assumption of exponentially declining tails is relaxed? Similar to a comment in Marckert and Mokkadem\cite{marckert_mokkadem2003}, we can argue that some number of finite moments and the assumption $\E L = 1$  should be sufficient to ensure that our convergence results stay intact. 
However, the infinite variance of the offspring distribution or power tails with $\E L < 1$ will likely lead to considerably different results.

 Another question is what happens in the intermediate regime when both $k$ and $n$ grow but $k/n \sim n^{-\beta}$ or $\big(k/n -(1 -\wh \nu)\big)\sim n^{-\beta}$ with $0 < \beta < 1$. The first case is when the tree has very small but still growing number of leaves. The second case is about the regime when the number of leaves in the tree is close to the maximal capacity that the offspring distribution can sustain.  The results in the paper of Labarbe and Marckert \cite{labarbe_marckert2007} about Bernoulli random walks suggest that even in this case the limit is Aldous' Brownian continuum random tree, although the scaling factor may be different from $n^{1/2}$. 

  Besides the scaling limit in the Gromov--Hausdorff topology, it could also be interesting to study the characteristics of the conditioned trees in these situations, such as their limiting degree sequences.

 \bigskip
 
 \textbf{Acknowledgements} This work was supported by Simons Foundation through the program ``Mathematics and Physical Sciences--Travel Support for Mathematicians'', Award ID 523587. 
 
 The author is indebted to an anonymous referee for a very careful report with many constructive comments that helped to improve this article.

\textbf{Data availability} Data sharing is not applicable to this article as no datasets were generated or analyzed during the current study.

\section*{Declarations}

\textbf{Conflict of interest}
 The author declares that he has no conflict of interest.

\appendix
\section{Proof of a Local Limit Theorem}
First, recall some basic definitions. A random variable $X$ has a \emph{lattice distribution} if, with probability 1, it takes values in a set $\{a + k h\}$ where $k \in \mathbb Z$ and $a$ and $h > 0$ are fixed numbers. Number $h$ is called a \emph{step} of the distribution. A step $h$ is called the maximal step if there is no $a_1$ and $h_1 > h$ so that all values of $X$ are in the set $\{a_1 + k h_1\}$. 

For every integer random variable $\xi$, we use $\xi^{(A)}$ to denote a truncated version of $\xi$; that is, $\xi^{(A)}$ has the distribution
\bal{
\P(\xi^{(A)} = j) = \P(\xi = j | \xi \leq A).
}

\begin{theo}
\label{theo_local_limit}
Let $\xi_1, \xi_2, \ldots $ be a sequence of i.i.d non-negative integer random variables. Assume that for some $\alpha > 0$, $\E e^{\alpha \xi_1} < \infty$, and let $\E \xi_1 = a$, $\Var (\xi_1) = \sigma^2$. Suppose that the maximal step of the distribution of $\xi_1$ equals $1$. Let $A_1, A_2, \ldots $ be a sequence of positive numbers such that $A_N = \Omega(N^\eps)$ for some $\eps > 0$. Then as $N \to \infty$ and uniformly over integer $n$, 
\bal{
\sigma \sqrt{N}\P(\xi_1^{(A_N)} + \ldots + \xi_N^{(A_N)} = n) - \frac{1}{\sqrt{2 \pi}} e^{-\frac{(n - aN)^2}{2 \sigma^2 N}} \to 0.
}
\end{theo}

Before starting the proof of the theorem, let us derive several useful lemmas.

\begin{lemma}
\label{lemma_max_step}
Suppose an integer-valued random variable $\xi \geq 0$ has maximal step $1$ Then there exists an $A_0$ such that for all $A > A_0$, $\xi^{(A)}$ has maximal step $1$.  
\end{lemma} 
\begin{proof}
Let $x_0 < x_1 < \ldots $ be the values of $\xi$ and let $Y \in \mathbb Z^+$ be the set of differences $x_{i} - x_{i - 1}$ for all $i = 1, 2, \ldots$. Then, the maximal step $h = \gcd(Y)$. Indeed, if $h$ is a step then all elements in $Y$ are divisible by $h$, so $h$ is a divisor of $\gcd(Y)$. Conversely, every $x_j$ can be written as $x_0 + \sum_{i = 1} (x_i - x_{i - 1})$ and therefore belongs to the set $x_0 + k \gcd(Y)$, so $\gcd(Y)$ is a step. 

By our assumption $gcd(Y) = 1$. Order elements of $Y$ as $y_1, y_2, \ldots $. Then for some finite $k$, $gcd(y_1, \ldots, y_k) = 1$. Take an $m$ such that all $y_i$ with $1 \leq i \leq k$ belong to the set of differences $x_{j} - x_{j - 1}$ with $1 \leq j \leq m$. Take $A_0 = x_m$. Then for every random variable $\xi^{(A)}$ with $A \geq A_0 = x_m$, the corresponding set of differences $Y^{(A)}$ contains all numbers $y_1, \ldots, y_k$ and therefore $gcd(Y^{(A)}) = 1$. This implies that the maximal step of $\xi^{(A)}$ equals 1.
\end{proof}
  
 Now let $\xi_i$, $A_N$, $\xi_i^{(A_N)}$ be as assumed in the statement of Theorem \ref{theo_local_limit}. In addition, let $a_N:= \E \xi_1^{(A_N)}$, $\sigma^2_N := \Var(\xi_1^{(A_N)})$, $m_3 := \E (|\xi_1 - a|^3)$, and $m_{3, N} :=  \E (|\xi_1^{(A_N)} - a_N|^3)$.
 
 \begin{lemma}
 \label{lemma_convergence_moments}
 Under assumptions on $\xi_i$, $A_N$, as in Theorem \ref{theo_local_limit}, for $N \to \infty$, we have:
 \begin{enumerate}[(a)]
 \item $(a_N - a)N^r \to 0$, for every $r > 0$, 
 \item $\sigma_N - \sigma \to 0$, 
 \item $m_{3, N} - m_{3} \to 0$.  
 \end{enumerate}
 \end{lemma}
 \begin{proof}
 Recall that for $A_N \to \infty$, and $\alpha > 0$, 
 \bal{
 \int_{A_N}^\infty x e^{-\alpha x} dx = O(A_N e^{-\alpha A_N}). 
 }
 Then, 
 \bal{
 N^r \E|\xi_1 - \xi_1 1(X \leq A_N)| = N^r \sum_{j > A_N} j P(\xi_1 = j)  
 = N^r O\Big(\sum_{j > A_N} j e^{-\alpha j} \Big), 
 }
 and by comparison with the integral we find that 
 \begin{equation}
 \label{truncation1}
  N^r \E|\xi_1 - \xi_1 1(X \leq A_N)|= O(N^r A_N e^{-\alpha A_N}) \to 0, 
  \end{equation}
  as $N \to \infty$, under assumption that $A_N = \Omega(N^\eps)$. 
  
  Then, 
  \begin{eqnarray}
   N^r \E|\xi_1 1(X \leq A_N) - \xi_1^{(A_N)}| &= N^r \sum_{j \leq A_N} j \P(\xi = j) \Big|1 - \frac{1}{\sum_{j \leq A_N} \P(\xi = j)}\Big| \notag
   \\
   \label{truncation2}
   &= N^r O\Big(\sum_{j > A_N} \P(\xi = j)\Big) = N^r O(e^{-\alpha N}) \to 0, 
  \end{eqnarray}
  as $N \to \infty$. 
  
 From (\ref{truncation1}) and (\ref{truncation2}) by the triangle inequality we find that 
 \bal{
  N^r \E|\xi_1 - \xi_1^{(A_N)}|\to 0,
 }
 which implies the first claim of the lemma. Clearly, the second and the third claims can be proved by similar arguments. 
 \end{proof}
 
 Finally, we need an estimate on characteristic functions. Let $f_N(t)$ be a characteristic function of $\xi_1^{(A_N)}$ and let 
 \begin{equation}
 \label{tilde_fN}
 \tilde f_N(t) = f_N(t) e^{-i a_N t}
 \end{equation}
  be the characteristic function of the centered r.v. $\xi_1^{(A_N)} - a_N$. 
 
 \begin{lemma}
 \label{lemma_estimate_fN}
  Under assumptions on sequences $(\xi_i)_{i = 1}^\infty$, $(A_N)_{N = 1}^\infty$, as in Theorem \ref{theo_local_limit}, we have:
  \begin{enumerate}[(a)]
  \item
  $
  \tilde f_N(t) = 1 - \frac{\sigma_N^2 t^2}{2} + O(|t|^3),
  $
  with the constant in $O$ term not depending on $N$. 
  \item
  there exists an $\eps > 0$ and $N_0$, such that if $|t| < \eps$ and $N > N_0$, then 
  \bal{
  |\tilde f_N(t)| 
  \leq e^{-\sigma^2 t^2/4}.
  } 
  \end{enumerate}
 \end{lemma}
 
 \begin{proof}
 By using Taylor's formula with remainder, we write:
 \bal{
 \Big|\tilde f_N(t) - \big(1 - \frac{\sigma_N^2 t^2}{2} \big)\Big| \leq \frac{m_{3, N} |t|^3}{6},
 } 
 which implies the first claim of the lemma in view of Lemma \ref{lemma_convergence_moments}(c). In addition, we have 
  \begin{equation}
  \label{estimate_tilde_fN}
 \Big|\tilde f_N(t) - \big(1 - \frac{\sigma^2 t^2}{2} \big)\Big| \leq
 \frac{|\sigma_N^2 - \sigma^2|t^2}{2} +  \frac{|m_{3, N} - m_3||t|^3}{6}+ \frac{m_{3} |t|^3}{6},
\end{equation} 
 By using Lemma  \ref{lemma_convergence_moments}(b) and (c), we can find $N_0$ and $\eps > 0$ such that for $|t| < \eps$ and $N > N_0$ the right-hand side of
 (\ref{estimate_tilde_fN}) is smaller than $\frac{\sigma^2 t^2}{4}$, and that $\frac{\sigma^2 t^2}{2} < 1.$
 Then,
 \bal{
 |\tilde f_N(t)| &= \Big|\tilde f_N(t) - \big(1 - \frac{\sigma^2 t^2}{2} \big)+ \big(1 - \frac{\sigma^2 t^2}{2} \big)\Big|
 \\
 &\leq  \frac{\sigma^2 t^2}{4} + (1 - \frac{\sigma^2 t^2}{2} \big) \leq 1 - \frac{\sigma^2 t^2}{4} \leq e^{-\sigma^2 t^2/4}.
 }
 \end{proof}
 
 Now we begin the proof of Theorem \ref{theo_local_limit}. Let 
 \bal{
 P_N(n) = \P(\xi_1^{(A_N)} + \ldots + \xi_N^{(A_N)} = n).
 }
 Since the characteristic function of $\xi_1^{(A_N)}$ is $f_N(t)$, the characteristic function of $S_N = \xi_1^{(A_N)} + \ldots + \xi_N^{(A_N)}$ equals 
 \bal{
 [f_N(t)]^N = \sum_{n = -\infty}^{\infty} P_N(n) e^{i t n}. 
 }
 By inversion formula, 
 \bal{
 P_N(n) = \frac{1}{2\pi} \int_{-\pi}^{\pi} e^{-i t n} [f_N(t)]^N\, dt.
 }
 By using  formula (\ref{tilde_fN}), parameterization $n = a N + z \sigma \sqrt{N}$, and the change of variable $x = t \sigma \sqrt N$, we rewrite this expression as 
  \bal{
 P_N(n) &= \frac{1}{2\pi} \int_{-\pi}^{\pi} e^{-i t (a - a_N) N} e^{-i t z \sigma \sqrt{N}}  [\tilde f_N(t)]^N\, dt
 \\
 &= \frac{1}{2\pi} \int_{-\pi \sigma \sqrt N}^{\pi \sigma \sqrt N} e^{-i x \frac{(a - a_N)}{\sigma} \sqrt{N}} e^{-i x z}  \Big[\tilde f_N\big(\frac{x}{\sigma \sqrt N}\big)\Big]^N\, \frac{dx}{\sigma \sqrt N}
 }
 Since, by inversion formula, 
 \bal{
 \frac{1}{\sqrt{ 2\pi }} e^{-z^2/2} 
 = \frac{1}{2\pi} \int_{-\infty}^{\infty} e^{-ixz - x^2/2} \, dx,  
 }
 we can write the difference 
 \bal{
 R_N = 2 \pi \Big(\sigma \sqrt N P_N(n) - \frac{1}{\sqrt{2 \pi}} e^{-z^2/2}\Big),
 }
 as a sum of four integrals:
 \bal{
 I_1 &= \int_{-A}^{A} e^{-i x z}\Big(e^{-i x \frac{(a - a_N)}{\sigma} \sqrt{N}}  \Big[\tilde f_N\big(\frac{x}{\sigma \sqrt N}\big)\Big]^N - e^{-x^2/2}\Big) \, dx, 
 \\
 I_2 &= - \int_{|x| \geq A} e^{-i x z - x^2/2}\, dx, 
 \\
 I_3 &= \int_{A \leq |x| \leq \eps \sigma \sqrt N} e^{-i x z}\Big(e^{-i x \frac{(a - a_N)}{\sigma} \sqrt{N}}  \Big[\tilde f_N\big(\frac{x}{\sigma \sqrt N}\big)\Big]^N\Big) \, dx,
 \\
 I_4 &= \int_{\eps \sigma \sqrt N \leq |x| \leq \pi \sigma \sqrt N} e^{-i x z}\Big(e^{-i x \frac{(a - a_N)}{\sigma} \sqrt{N}}  \Big[\tilde f_N\big(\frac{x}{\sigma \sqrt N}\big)\Big]^N\Big) \, dx,
 }
 where $A$ and $\eps$ will be chosen later. 
 
 In order to show that $R_N \to 0$ for $N \to \infty$, we take an arbitrary $\delta > 0$ and show that by a choice of sufficiently large $N$ the difference can be made smaller than $\delta$. 
 
 Since $|I_2| \leq \int_{|x| \geq A} e^{-x^2/2} \, dx$, it can be made arbitrarily small by a choice of a sufficiently large $A$. 
 
 Next, we can choose a sufficiently small $\eps$ such that for all $N \geq N_0$,  Lemma  \ref{lemma_estimate_fN}(b) applies and for $|x| \leq \eps \sigma \sqrt N$, 
 \bal{
\Big| \tilde f_N\big(\frac{x}{\sigma \sqrt N}\big)\Big| \leq \exp \Big( - \frac{\sigma^2}{4} \big(\frac{x}{\sigma \sqrt N}\big)^2\Big) = e^{- x^2/(4N)}.
 }
 Then, we can estimate $I_3$:
 \bal{
 |I_3|\leq \int_{A \leq |x| \leq \eps \sigma \sqrt N} \Big|\tilde f_N\big(\frac{x}{\sigma \sqrt N}\big)\Big| ^N \, dx \leq \int_{A \leq |x|} e^{-x^2/4}\, dx,
 }
 which can be made arbitrarily small by a choice of a sufficiently large $A$. 
 
 Fix $A$ and $\eps$ so that $|I_2|< \delta/4$ and $|I_3 |<\delta/4$ for $N > N_0$.
 
 Next, we address $I_1$. Let 
 \bal{
 g_N(x) := \Big[ e^{- i\frac{a - a_N}{\sigma \sqrt N} x } \tilde f_N\big(\frac{x}{\sigma \sqrt N}\big)\Big]^N,
 }
 and note that this is the characteristic function of $(S_N - a N)/(\sigma \sqrt N)$.  We claim that the distribution of this random variable converges to the normal distribution with parameters $(0, 1)$. Indeed, by using  Lemma \ref{lemma_estimate_fN}(a) and Lemmas  \ref{lemma_convergence_moments}(a) and (b), we find that for a fixed $x$ and $N \to \infty$, 
 \bal{
 \log g_N(x) &= N  \Big[ - i\frac{a - a_N}{\sigma \sqrt N} x + \log\Big( \tilde f_N\big(\frac{x}{\sigma \sqrt N}\big)\Big)\Big]
 \\
 &= - i\frac{a - a_N}{\sigma} \sqrt N x + N \log\Big( 1 - \frac{\sigma_N^2 x^2}{2\sigma^2 N} + o\big(\frac{1}{N})\Big)
 \\
 &= -\frac{x^2}{2} + o(1), 
 }
 which implies that $g_N(x) \to e^{-x^2/2}$ for every fixed $x$, and therefore, the distribution of $(S_n - a N)/(\sigma \sqrt N)$ converges weakly to the standard normal distribution. Now, we invoke the fact that the weak convergence of probability measures implies the uniform convergence of characteristic functions on every finite interval. This implies that for every $A$, integral $I_1$ converges to zero as $N\to \infty$.  
 
 It remains to estimate integral $I_4$. We have
 \bal{
 |I_4| \leq \int_{\eps \leq \frac{|x|}{\sigma \sqrt N} \leq \pi} \Big|\tilde f_N \big(\frac{x}{\sigma \sqrt N}\big)\Big|^N \, dx 
 = \sigma \sqrt N \int_{\eps \leq |t| \leq \pi} \big| f_N(t)|^N \, dt. 
 }
 By Lemma \ref{lemma_max_step}, the maximal step of $\xi_1^{(A_N)}$ equals $1$ for all sufficiently large $N$, which implies that 
 \bal{
 \max_{\eps \leq |t| \leq \pi} |f_N(t)| = q < 1. 
 } 
 Therefore, $|I_4| \leq \sigma \sqrt N 2 \pi q^N$ and $I_4 \to 0$ as $N \to \infty$.
 
 The estimates of integrals $I_1$ and $I_4$ show that there exists $N_1$ such that $|I_1| \leq \delta/4$ and $|I_4| \leq \delta/4$ for $N > N_1$. Hence, the difference $R_N$ for $N \to \infty$ approaches $0$ uniformly over integer $n$.


\begin{thebibliography}{10}

\bibitem{aldous91a}
David Aldous.
\newblock The continuum random tree \mbox{I}.
\newblock {\em Annals of Probability}, 19(1):1 -- 28, 1991.

\bibitem{aldous91b}
David Aldous.
\newblock The continuum random tree \mbox{II}: an overview.
\newblock In M.~T. Barlow and N.~H. Bingham, editors, {\em Proceedings of the
  Durham Symposium on Stochastic Analysis, 1990}, pages 23 -- 70. Cambridge
  University Press, 1991.

\bibitem{aldous93}
David Aldous.
\newblock The continuum random tree \mbox{III}.
\newblock {\em Annals of Probability}, 21(1):248--289, 1993.

\bibitem{billingsley1968}
Patrick Billingsley.
\newblock {\em Convergence of Probability Measures}.
\newblock John Wiley and Sons, Inc., first edition, 1968.

\bibitem{broutin_marckert2014}
N.~Broutin and J.~F. Marckert.
\newblock Asymptotic of trees with a prescribed degree sequence.
\newblock {\em Random Structures and Algorithms}, 44:290--316, 2014.

\bibitem{burago_burago_ivanov2001}
Dmitri Burago, Yuri Burago, and Sergei Ivanov.
\newblock {\em A Course in Metric Geometry}, volume~33 of {\em Graduate Studies
  in Mathematics}.
\newblock American Mathematical Society, 2001.

\bibitem{dershowitz_zaks1990}
Nachum Dershowitz and Shmuel Zacks.
\newblock The cycle lemma and some applications.
\newblock {\em European Journal of Combinatorics}, 11:35--40, 1990.

\bibitem{devroye2012}
Luc Devroye.
\newblock Simulating size-constrained \mbox{G}alton--\mbox{W}atson trees.
\newblock {\em SIAM J. of Computing}, 41(1):1 -- 11, 2012.

\bibitem{duquesne2003}
Thomas Duquesne.
\newblock A limit theorem for the contour process of conditioned
  \mbox{G}alton--\mbox{W}atson trees.
\newblock {\em Annals of Probability}, 31(2):996--1027, 2003.

\bibitem{duquesne_legall2002}
Thomas Duquesne and Jean-Fran\c{c}ois Le\phantom{ }Gall.
\newblock Random trees, \mbox{L}\'{e}vy processes and spatial branching
  processes.
\newblock {\em Ast\'{e}risque}, 281, 2002.

\bibitem{dvoretsky_motzkin1947}
A.~Dvoretzky and Th. Motzkin.
\newblock A problem of arrangements.
\newblock {\em Duke Mathematical Journal}, 14:305--313, 1947.

\bibitem{evans_pitman_winter2006}
Steven~N. Evans, Jim Pitman, and Anita Winter.
\newblock Rayleigh processes, real trees, and root growth with re-grafting.
\newblock {\em Probability Theory and Related Fields}, 134:81--126, 2006.

\bibitem{feller71}
William Feller.
\newblock {\em An Introduction to Probability Theory and Its Applications},
  volume~2.
\newblock John Wiley and Sons, 2 edition, 1971.

\bibitem{flajolet_sedgewick2009}
Philippe Flajolet and Robert Sedgewick.
\newblock {\em Analytic Combinatorics}.
\newblock Cambridge University Press, 2009.

\bibitem{haas_miermont2012}
B\'{e}n\'{e}dicte Haas and Gr\'{e}gory Miermont.
\newblock Scaling limits of \mbox{M}arkov branching trees with applications to
  \mbox{G}alton--\mbox{W}atson and random unordered trees.
\newblock {\em Annals of Probability}, 40(6):2589--2666, 2012.

\bibitem{janson2012}
Svante Janson.
\newblock Simply generated trees, conditioned \mbox{G}alton--\mbox{W}atson
  trees, random allocations and condensation.
\newblock {\em Probability Surveys}, 9:103--252, 2012.

\bibitem{kolchin84}
V.~F. Kolchin.
\newblock {\em Random Mappings}.
\newblock (English translation: Optimization Software, New York, 1986). Nauka,
  Moscow, russian edition, 1984.

\bibitem{kortchemski2012}
Igor Kortchemski.
\newblock Invariance principles for \mbox{G}alton--\mbox{W}atson trees
  conditioned on the number of leaves.
\newblock {\em Stochastic Processes and Their Applications}, 122:3126--3172,
  2012.

\bibitem{kortchemski2013}
Igor Kortchemski.
\newblock A simple proof of \mbox{D}uquesne's theorem on contour processes of
  conditioned \mbox{G}alton--\mbox{W}atson trees.
\newblock In C.~Donati-Martin, editor, {\em Seminaire de Probabilities XLV},
  volume 2078 of {\em Lecture Notes in Mathematics}, pages 537 -- 558. Springer
  International Publishing, 2013.

\bibitem{labarbe_marckert2007}
Jean-Maxime Labarbe and Jean-Fran\c{c}ois Marckert.
\newblock Asymptotics of \mbox{B}ernoulli random walks, bridges, excursions and
  meanders with a given number of peaks.
\newblock {\em Electronic Journal of Probability}, 12:229--261, 2007.

\bibitem{lei2019}
Tao Lei.
\newblock Scaling limit of random forests with prescribed degree sequences.
\newblock {\em Bernoulli}, 25:2409--2438, 2019.

\bibitem{legall2005}
Jean-Fran\c{c}ois Le\phantom{ }Gall.
\newblock Random trees and applications.
\newblock {\em Probability Surveys}, 2:245--311, 2005.

\bibitem{legall_lejan98}
Jean-Fran\c{c}ois Le\phantom{ }Gall and Yves Le\phantom{ }Jan.
\newblock Branching processes in \mbox{L}\'evy processes: the exploration
  process.
\newblock {\em Annals of Probability}, 26:213--252, 1998.

\bibitem{marckert_mokkadem2003}
Jean-Fran\c{c}ois Marckert and Abdelkader Mokkadem.
\newblock The depth-first processes of \mbox{G}alton--\mbox{W}atson trees
  converge to the same \mbox{B}rownian excursion.
\newblock {\em Annals of Probability}, 31(3):1655--1678, 2003.

\bibitem{oht2021}
Gabriel~Berzunza Ojeda, Cecilia Holmgren, and Paul Th\'evenin.
\newblock Convergence of trees with a given degree sequence and of their
  associated laminations.
\newblock https://arxiv.org/abs/2111.07748, 2021.

\bibitem{petrov75}
V.~V. Petrov.
\newblock {\em Sums of independent random variables}.
\newblock Springer Berlin Heidelberg New York, 1975.

\bibitem{petrov95}
V.~V. Petrov.
\newblock {\em Limit theorems of probability theory}, volume~4 of {\em Oxford
  Studies in Probability}.
\newblock The Clarendon Press, Oxford, 1995.

\bibitem{pitman2002}
Jim Pitman.
\newblock {\em Combinatorial Stochastic Processes}, volume 1875 of {\em Lecture
  Notes in Mathematics}.
\newblock Springer, 2002.

\bibitem{rizzolo2015}
Douglas Rizzolo.
\newblock Scaling limits of \mbox{M}arkov branching trees and
  \mbox{G}alton--\mbox{W}atson trees conditioned on the number of vertices with
  out-degree in a given set.
\newblock {\em Annales de l'Institut Henri Poincar{\'e}. Probabilit{\'e}s et
  Statistiques}, 51(512 - 532), 2015.

\bibitem{thevenin2020}
Paul Th\'evenin.
\newblock Vertices with fixed outdegrees in large \mbox{G}alton--\mbox{W}atson
  trees.
\newblock {\em Electronic Journal of Probability}, 25:1--25, 2020.

\bibitem{vervaat79}
Wim Vervaat.
\newblock A relation between \mbox{B}rownian bridge and \mbox{B}rownian
  excursion.
\newblock {\em Annals of Probability}, 7(1):143 -- 149, 1979.

\end{thebibliography}

\end{document}